\documentclass[french]{amsart}
\usepackage[utf8]{inputenc} 
\usepackage[T1]{fontenc} 
\usepackage[french]{babel}
\usepackage[widespace,upright]{fourier}
\usepackage{appendix}
\title[Sur les représentations ultramétriques]{Résultat géométrique sur les représentations de groupes réductif sur un corps ultramétrique}
\author{Rodolphe \textsc{Richard}}

\newcommand{\kk}{\mathbf{k}}
\newcommand{\abs}[1]{{\left|{#1}\right|}}
\usepackage{mathrsfs}
\newcommand{\Ik}{\mathscr{I}_\kk}
\newcommand{\an}{\textrm{an}}
\newcommand{\SL}{\mathrm{SL}}
\newcommand{\GL}{\mathrm{GL}}
\newcommand{\Ad}{\mathrm{Ad}}
\newcommand{\lie}[1]{{\mathfrak{#1}}}
\newcommand{\gl}{\lie{gl}}
\newcommand{\Id}{\mathrm{Id}}
\newtheorem{theoreme}{Théorème}[section]
\newcommand{\Nm}[1]{{\left\|{#1}\right\|}}
\usepackage{enumitem}

\newcommand{\R}{\mathbf{R}}
\newcommand{\Z}{\mathbf{Z}}
\newcommand{\z}{\lie{z}}
\newcommand{\Hom}{\mathrm{Hom}}
\newcommand{\tens}{\mathop{\otimes}}
\newtheorem{proposition}{Proposition}[section]
\newcommand{\NM}[1]{{\left\vert\kern-0.25ex\left\vert\kern-0.25ex\left\vert #1 
    \right\vert\kern-0.25ex\right\vert\kern-0.25ex\right\vert}}
\newtheorem{corollaire}{Corollaire}[section]
\newtheorem{lemme}{Lemme}[section]
\newcommand{\Aff}{\mathbf{Aff}}
\begin{document}
\maketitle
\setcounter{tocdepth}{0}
\tableofcontents

\section*{Notes d'Édition}

Ce texte est une réédition d'un article publié au sein de la thèse de doctorat~\cite{these} de l'auteur, datant de 2009. Il s'agit de son Chapitre~\uppercase\expandafter{\romannumeral 5\relax}.  Hormis la présente section, cette édition est pour l'essentiel un \emph{verbatim}, avec quelques corrections mineures, et une correction majeure, ne concernant que la caractéristique positive, que nous discutons~p.~\pageref{Erratum}; espérons qu'en revanche cette édition n'introduit que le moins possible de maladresses.

 Cette édition allège également l'annexe de deux sections, de rappels sur les~\emph{variétés algébriques} et \emph{groupes algébriques}; le lecteur, si besoin, saura les retrouver aisément dans d'autres références, par ex.~\cite{Bor91} et~\cite{PR94}, qui sont standard, exhaustives, et bien meilleures. On abrège enfin les rappels sur les normes ultramétriques et remarque que l'on peut probablement se passer de recourir à l'énoncé~\ref{LemmeC1}.

\subsection*{\emph{Leitfaden}}\label{Leitfaden} À quelle fin ce texte, dans quel contexte s'inscrit-il? C'est un analogue \emph{ultramétrique} de~\cite{RS09}. Ce dernier traite du cas \emph{archimédien}. Une version de~\cite{RS09} est contenue dans l'article publié~\cite{RichardShah}, qui améliore ce résultat, en y développant de plus des «~conséquences~», plus importantes du point de vue dynamique. Ce développement, postérieur à~\cite{these}, une astuce en fait, y est écrit de telle sorte que, connaissant les résultats ultramétriques ici présents, on obtient par la preuve mot-à-not l'analogue ultramétrique des résultats de~\cite{RichardShah}, ces «~conséquences~» y compris.

Des applications dynamiques de~\cite{RS09} et du présent texte sont développées au Chapitre~\uppercase\expandafter{\romannumeral 6\relax}
 de la thèse sus-citée.

Les résultats de~\cite{RichardShah} dans leur intégralité (et pas seulement de~\cite{RS09}),
qui sont archimédiens, et leur analogue ultramétrique, qui repose sur le présent article, sont à la base du travail de~\cite{RichardZamojski} en dynamique homogène, qui dépasse de loin ceux de ce Chapitre~\uppercase\expandafter{\romannumeral 6\relax}.

Les travaux de~\cite{RichardZamojski} sont à leur tour la base du travail~\cite{RichardYaffaev} établissant la conjecture d'André-Pink, dans certains cas seulement, mais, en revanche, sous une forme améliorée: on y détaille l’équidistribution et l’adhérence topologique en plus de l’adhérence de Zariski.

Pour résumer, le texte ici présent est un des socles sur lequel sont bâtis les travaux que nous venons de passer en revue,~\cite{RichardShah}, \cite{RichardZamojski}, \cite{RichardYaffaev}. Ce texte n'était, jusqu'à présent, pas diponible dans publication à revue par des pairs. Son usage dans dans~\cite{RichardYaffaev} est l'occasion de le rééditer et d'effectuer cette publication.

Nimish \textsc{Shah} est coauteur de~\cite{RS09} et~\cite{RichardShah}; Tomasz \textsc{Zamojski} l'est pour~\cite{RichardZamojski}; Andreï \textsc{Yaffaev} pour~\cite{RichardYaffaev}.

\subsection*{Mise en perspective} À l'époque de ce résultat, en 2009, les outils utilisés, immeubles de Bruhat-Tits en relation avec les espaces analytiques de Berkovich, m'apparaissaient invraisemblablement sophistiqués en mesure
du résultat ambitionné. Les outils utilisés, dans leur forme aboutie, étaient même assez neufs, contemporains. Avec le recul, les résultats ici obtenus peuvent se voir comme une question de \emph{stabilité}, au sens de la théorie géométrique des invariants de Mumford, mais dans un contexte (ultra)métrique plutôt que géométrique: plutôt que de se demander, dans une représentation linéaire, si une orbite donnée est adhérente de Zariski en l'origine~$0$, il s'agit de savoir si elle s'approche beaucoup de~$0$, au sens métrique. Les questions de stabilité s'étudient naturellement dans espaces de drapeaux «~\emph{flag varieties}~» chez Mumford, autrement connus comme immeubles sphériques de Tits. L'analogue pour les corps ultramétriques est l'immeuble de Bruhat-Tits; l'immeuble sphérique, ou plutôt sa forme vectorielle, est l'immeuble de Bruhat-Tits sur ce corps muni de la structure ultramétrique discrète. 

Une forme métrique, mais archimédienne, de la stabilité était déjà étudiée par la théorie de Mumford-Ness (voir édition récentes de~\cite{GIT}). Une forme ultramétrique, puis arithmétique, l'était par~\cite{Burnol}.

Cela justifie \emph{a posteriori} les outils utilisés ici, voire ceux de~\cite{RS09}. Pour conclure, mentionnons que c'est ce point de vue «~stabilité~» d'où provient l'astuce utilisée dans l’article~\cite{RichardShah}.

\section*{Introduction originelle}
Cet article adapte les résultats de \cite{RS09}\footnote{Dont l'article publié~\cite{RichardShah} est une version augmentée. Voir \emph{Leifaden} en page~\pageref{Leitfaden}.} au cas d'un groupe algébrique semi\-sim\-ple~$G$ sur un corps local~$\kk$ muni d'une valeur absolue ultramétrique~$\abs{-}$.

Nous suivons, dans les grandes lignes, la méthode développée dans~\cite{RS09} pour le contexte archimédien (dont nos Propositions~\ref{Prop31} et~\ref{Prop34} reprennent certains arguments). Pour pallier l'absence du théorème de décomposition de Mostow (\cite{Mos55}) pour~$G(\kk)$ ainsi que la propriété de convexité de l'application exponentielle, qui n'est plus partout définie, nous considérons le plongement~$\Theta:\Ik(G)\to G^\an$ de l'\emph{immeuble de Bruhat-Tits}~$\Ik(G)$ de~$G$ sur~$\kk$ dans l'\emph{espace analytique}~$G^\an$, au sens de Berkovich, associé à~$G$.

Notre démonstration repose en effet sur le Théorème~\ref{TheoF3} de l'annexe, qui se base sur les travaux~\cite{RTW09} de Bertrand Rémy, Amaury Thuillier et Anette Werner. Ces auteurs généralisent une construction du chapitre 5 de~\cite{Ber90}, où V.~Berkovich se restreint aux groupes de Chevalley (semi-simples déployés). Le Théorème~\ref{TheoF3} met à profit les propriétés de convexité dans~$\Ik(G)$, et remplace la propriété de convexité de l'application exponentielle de~\cite{RS09} par la convexité de fonctions de la forme~$x\mapsto \abs{f(x)}$ (Proposition~\ref{Prop36}) lorsque~$f$ est dans~$\kk[G]$, une fonction régulière.

La décomposition de Mostow est remplacée par la décomposition moins
précise du Théorème~\ref{TheoF1} et sa conséquence en la Proposition~\ref{Prop21}. Pour obtenir l'énoncé~\ref{Prop21}, notre démonstration utilise l'existence de points fixes pour l'action
de groupes d'isométries compacts sur les immeubles de Bruhat-Tits. Ces propriétés
découlent de l'existence de métriques hyperboliques, et justifient le choix
de la géométrie ultramétrique au sens de Berkovich, plutôt que rigide, qui permet de considérer des
espaces métriques complets.

Que Bertrand Remy, Amaury Thuillier et Georges Tomanov reçoivent ici mes
remerciements pour leur accueil chaleureux et leur conversation enrichissante à
l'occasion de mon déplacement à l'institut Camille Jordan de l'université Lyon 1
Claude Bernard.

C'est en côtoyant, à l'IRMAR, Antoine Chambert-Loir, Antoine Ducros et Jérôme
Poineau que j'ai pu me familiariser avec les espaces de Berkovich. Que cet article
leur témoigne de ma reconnaissance.

C’est la relecture par Emmanuel Breuillard qui a conduit à l’erratum.

\addtocontents{toc}{\protect\setcounter{tocdepth}{1}}
\section{Hypothèses et Énoncé}

Nous convenons qu'un \emph{groupe algébrique linéaire}, et plus généralement un sous-groupe algébrique linéaire, est supposé \emph{affine de type fini}, \emph{réduit} et \emph{connexe}.
La nécessité de ces hypothèse n'a pas été vérifiée : notons que, 
comme notre résultat principal ne concerne que les groupes de points rationnels, il peut s'appliquer aux groupes non réduits, quitte à passer au sous-groupe réduit associé. Remarquons aussi que, dans un groupe algébrique linéaire, le centralisateur d'un sous-groupe algébrique linéaire n'est pas toujours réduit, ni toujours connexe. Par exemple, en caractéristique non nulle~$p$, le centre de~$\SL(p)$ est connexe, de dimension~$0$ mais a une algèbre de Lie
non nulle. En caractéristique~$0$ le centre de~$\SL(2)$ n’est pas connexe.

Soit~$G$ un groupe algébrique linéaire sur un corps~$\kk$. Étant donnée une représentation linéaire de degré fini~$\rho : G \rightarrow \GL(V)$, on notera~$\Ad_\rho : G \rightarrow GL (\gl(V ))$ l'action par conjugaison~$G$ sur~$\gl(V)$.
\begin{equation}\label{eq1}
\text{Pour~$g$ dans~$G$, et un endomorphisme~$e$ de~$V$, on a~$\Ad_\rho(g)=\rho(g)e\rho(g)^{-1}$.}
\end{equation}
Pour tout sous-groupe algébrique linéaire~$H$ de~$G$, on notera~$C_H(\Ad_\rho)$ l'espace vectoriel sur~$\kk$ engendré par les coefficients matriciels de l'action de~$H$ sur~$\gl(V)$. L'espace~$C_H(\Ad_\rho)$ est formé de fonctions régulières sur~$H$. Étant donné un ensemble~$\Omega$ de points de~$H$, nous considérons la propriété suivante.
\begin{equation}\tag{\text{$\ast$}}\label{*}
\text{Tout coefficient matriciel de~$\Ad_\rho$ qui s'annule sur~$\Omega$ s'annule en fait sur~$H$.} 
\end{equation}
Cette propriété est notamment vérifiée si~$\Omega$ est Zariski dense dans~$H$. Mettons en exergue deux autres conditions sur le $H$-module~$V$.
\begin{multline}\tag{\text{$\ast\ast$}}\label{**}
	\text{L'action~$\rho$ de~$H$ sur~$V$ est telle que~$V^H$, le plus grand sous-module de}\\
	\shoveleft{\text{points fixes, a un unique supplémentaire $H$-stable.}}
\end{multline}
\begin{multline}\tag{\text{$\ast\ast^\prime$}}\label{**'}
	\text{En surcroît de~\eqref{**}, ce supplémentaire de~$V^H$, comme représentation }\\ \shoveleft{\text{de~$H$, n'a pas de quotient isomorphe à la représentation triviale.}}
\end{multline}
La condition \eqref{**} revient à supposer que~$V$ ne contient pas d'extension non triviale \emph{de} la représentation triviale, et la condition~\eqref{**'} que~$V$ ne contient pas non
plus d'extension non triviale \emph{par} la représentation triviale. Les conditions~\eqref{**}~et~\eqref{**'} sont automatiquement vérifiées si le $H$-module~$V$ est semi-simple. C'est le cas si~$H$ est réductif et le corps~$\kk$ de caractéristique nulle.

\subsubsection*{Remarqes:}{~}
\newline\noindent
\textit{i)}\label{Remarquei}
 Lorsque la condition~\eqref{**} est vérifiée, le supplémentaire~$H$-stable de~$V^H$ est le noyau d'un \emph{unique} projecteur~$H$-équivariant de~$V$ sur~$V^H$ (le \emph{projecteur de
Reynolds}, \emph{confer} \cite{Dem76}). Étant donné un morphisme $H$-équivariant~$\Phi : V \rightarrow W$ entre deux $H$-mo\-du\-les~$V$ et~$W$ satisfaisant~\eqref{**}, ce morphisme commute aux projecteurs sur le lieu fixe si~$V$ satisfait la condition~\eqref{**'}.
\newline
\textit{ii)}\label{Remarqueii}
En revanche, en caractéristique non nulle~$p$, l'action adjointe de ~$\GL(p)$ sur~$\gl(p)$ ne vérifie pas la condition~\eqref{**}. Le sous-espace fixe est la droite~$\kk\cdot\Id$ formée des homothéties. Or cette droite n'a pas de supplémentaire stable: c'est déjà le cas 
pour l'action des matrices de permutation sur le sous-espace diagonal. Lorsque~$p$ est impair, l'action induite de~$\GL(p)$ sur~$\frac{\lie{gl}(p)}{\kk\cdot\Id}$ satisfait la condition~\eqref{**} car il n'y a pas d'élément invariant non nul (cf.~\cite{Bou60},~\S6, Exercice~24). Elle ne satisfait pas la condition~\eqref{**'} car l'action quotient~$\left.\frac{\lie{gl}(p)}{\kk\cdot\Id}\middle/\frac{\lie{sl}(p)}{\kk\cdot\Id}\right.$ est triviale.

Il nous faut encore considérer la propriété suivante\footnote{N.d.É. Voir \emph{Erratum} p.~\pageref{Erratum}.}, de Richardson.
\begin{equation}\tag{\text{$\ast\ast\ast$}}\label{***}
\text{Le sous-groupe~$H$ de~$G$ est «~fortement réductif dans~»~$G$.}
\end{equation}
En caractéristique nulle elle est  satisfaite si et seulement si~$H$ est réductif. Un propriété aisément vérifiable qui implique~\eqref{***} est la suivante, dite~$H$ est «~réductif dans~» $G$ selon les termes de~\cite{RichardShah}.
\begin{equation}\tag{\text{$\ast\ast\ast^\prime$}}\label{***'}
\text{L'action de~$H$ action sur~$\lie{g}$ est semi-simple.}
\end{equation}

Notre résultat principal est le suivant. Pour motiver cet énoncé nous ren\-voy\-ons à~\cite{RS09} et~\cite{Ric09}.
\begin{theoreme}\label{Theo11}

Soit~$(\kk,\abs{-})$ un corps local normé ultramétrique, soit~$G$ un groupe linéaire algébrique semisimple connexe sur~$\kk$, et soit~$H$ un sous-groupe algébrique linéaire fortement réductif dans~$G$. Notons~$Z$ le centralisateur de~$H$ dans~$G$. Alors il existe une partie~$Y$ de~$G(\kk)$, fermée pour la topologie ultramétrique, et telle que
\begin{enumerate}[label=\textrm{\arabic*}.]
\item[\textup{1.}] d'une part on ait~$G(\kk)=Y\cdot Z(\kk)$,
\item[\textup{2.}] d'autre part, étant donnés
\begin{itemize}
\item une représentation linéaire~$\rho:G\rightarrow \GL(V)$ de degré fini et définie sur~$\kk$, telle que les~$H$-modules~$\gl(V)$ et~$C_H(\Ad_\rho)$ satisfassent~\eqref{**}, et que~$\lie{gl}(V)$ satisfasse~\eqref{**'},
\item une partie non vide~$\Omega$ de~$H(\kk)$ ayant la propriété~\eqref{*},
\item une norme~$\Nm{-}$ sur~$V$, supposée homogène relativement à~$\abs{-}$,
\end{itemize}
il existe une constante~$c>0$ telle que
\begin{equation}\label{eq2}
\forall y\in Y,~\forall v\in V,~\sup_{\omega\in\Omega}\Nm{\rho(y\cdot \omega)(v)}\geq\Nm{v}/c.
\end{equation}
\end{enumerate}
\end{theoreme}

Fixons~$\rho$. L'inégalité~\eqref{eq2} est vérifiée pour toute constante~$c$ lorsque le vecteur~$v$
est nul. Le théorème est donc vérifié, avec~$Y = G(\kk)$, si~$V$ est de dimension nulle.
\emph{Dorénavant nous supposerons que la représentation~$\rho$ a un degré non nul}. \label{constantes coeffs}	En particulier les fonctions constantes sont des coefficients matriciels. En effet, tout coefficient diagonal de~$g\mapsto \Ad_\rho(g)\Id_V$ vaut la constante~$1$. Ainsi~$C_G(\Ad_\rho)$ et~$C_H(\Ad_\rho)$ seront non nuls. Dans ce cas, la non vacuité de~$\Omega$ découle de la
condition~\eqref{*}.

Remarquons que pour établir la formule~\eqref{eq2}, on peut remplacer~$\Omega$ par un sous-ensemble~$\Omega_b$ de~$\Omega$, car cela a pour effet de diminuer le membre de gauche de l'inégalité, sans modifier le membre de droite. Montrons que, comme~$\Omega$ satisfait~\eqref{*}, il
existe un sous-ensemble fini~$\Omega_b$ de~$\Omega$ satisfaisant la propriété~\eqref{*}.

\begin{proof}
La condition~\eqref{*} signifie que lorsque~$\omega$ décrit~$\Omega$, les morphismes d'é\-va\-lu\-a\-tion~$f \mapsto f (\omega)$, définis sur~$C_H(\Ad_\rho)$, engendrent~$C_H(\Ad_\rho)^\vee$, le dual algébrique de l’epace~$C_H(\Ad_\rho)$. 
Comme V est de dimension finie,~$C_H(\Ad_\rho)$, qui
est un quotient de~$\gl(V ) \otimes \gl(V )^\vee$, est de dimension finie.
Il suffit donc d'extraire
de la famille génératrice précédente une base, nécessairement finie, et de choisir pour~$\Omega_b$ un sous-ensemble fini de~$\Omega$ paramétrant cette base.
\end{proof}

\emph{Dorénavant~$\Omega$ sera supposé borné dans~$H(\kk)$.}

Notre démonstration utilise les énoncés~\ref{TheoF3} et~\ref{TheoF1} de l'annexe, à laquelle nous renvoyons pour les définitions et conventions utilisées, en particulier concernant la notion de convexité telle que définie dans la section~\ref{sectionF}. Pour plus d'approfondissement,
on pourra également consulter~\cite{RTW09}, ainsi que~\cite[chapitre 5]{Ber90} pour le cas des groupes semisimples déployés (« de Chevalley »).

Dans la section suivante, nous rappelons la situation et fixons les notations
utilisées jusque la fin de la démonstration, soit les sections~\ref{section2}, \ref{section3} et~\ref{section4}. Nous y explicitons en particulier la
partie~$Y $. La première conclusion du Théorème~\ref{Theo11} résulte de la Proposition~\ref{Prop21},
que nous déduisons de l'énoncé~\ref{TheoF1}. Nous énonçons également, avec la Proposition~\ref{Prop22}, une variante effective de la seconde conclusion du Théorème~\ref{Theo11} pour la
partie~$Y$ construite. 

Dans la section~\ref{section3}, nous réunissons quelques énoncés indépendants qui seront
utilisés dans la démonstration de la Proposition~\ref{Prop22}. Les énoncés~\ref{Prop31} à~\ref{Prop34} reprennent des arguments de~\cite{RS09}. La démonstration de la Proposition~\ref{Prop36}
repose sur l'énoncé~\ref{TheoF3}. Le cœur de la démonstration de la
Proposition~\ref{Prop22} occupe la section~\ref{section4}.

\section{Notations}\label{section2}

Rappelons la situation. Nous désignons par~$\kk$ un corps local muni d'une valeur
absolue ultramétrique~$\abs{-}$, par~$G$ un groupe algébrique linéaire semi-simple sur~$\kk$,
par~$H$ un sous-groupe algébrique linéaire fortement réductif, et notons~$Z_G(H)$ le centralisateur de~$H$ dans~$G$, vu
comme groupe algébrique affine non nécessairement réduit. Notons que~$Z_G (H)$ n'interviendra toutefois que via son groupe~$Z_G (H)(\kk)$ des points rationnels.

Nous nous sommes fixés~$\rho : G \rightarrow \GL(V )$, une représentation linéaire de~$G$ de degré fini non nul et définie sur~$\kk$ et~$\Nm{-} : V \rightarrow \R$ une norme $(\kk,\abs{-})$-homogène sur~$V$. Nous notons~$\gl(V )$ l'algèbre de Lie des endomorphismes de~$V$, et~$\Ad_\rho$ la
représentation adjointe~\eqref{eq1} de~$G$ sur~$\gl(V )$. Le sous-$G$-module de la représentation régulière~$\kk[G]$ engendré par les coefficients matriciels de~$\Ad_\rho$ est noté $C_G(\Ad_\rho)$, et le~$H$-module formé de
la restriction à~$H$ de ces fonctions régulières est noté~$C_H(\Ad_\rho)$.

Nous désignons par~$\Omega$ une partie bornée non vide de~$H(\kk)$ sur laquelle aucune
fonction régulière non nulle sur~$H$ issue de~$C_H(\Ad_\rho)$ ne s'annule identiquement.

Notons~$\z$ le centralisateur de~$H$ dans~$\gl(V )$. D'après l'hypothèse~\eqref{**} pour~$\Ad_\rho$, il existe un unique projecteur $H$-équivariant de~$\gl(V )$ sur~$\z$ (cf. Remarque~\emph{i)} en page~\pageref{Remarquei}).
Comme~$V$ est de dimension non nulle, les coefficients diagonaux de~$h\mapsto \Ad_\rho(h)(\Id_V )$ forment un coefficient matriciel constant non nul de~$\Ad_\rho$. En outre le sous-module fixe de~$C_H(\Ad_\rho)$ est contenu dans celui de~$\kk[H]$, qui est aussi donné par les fonctions constantes. Donc le sous-module fixe de~$C_H(\Ad_\rho)$ est le sous-module formé des fonctions
constantes, et s'identifie à~$\kk$, muni de la représentation triviale. Utilisant l'hypothèse~\eqref{**} pour~$C_H(\Ad_\rho)$ nous obtenons un projecteur~$H$-équivariant de~$C_H(\Ad_\rho)$ sur~$\kk$. En vertu de l'hypothèse~\eqref{**'} pour~$\Ad_\rho$, tout morphisme~$H$-équivariant~$\gl(V)\rightarrow C_H(\Ad_\rho)$ commute aux projecteurs~$\pi_\kk$ et~$\pi_\z$ (cf. Remarque~\emph{i)} en page~\pageref{Remarquei}).

Nous nous fixons un tore déployé maximal~$T$ de~$G$ sur~$\kk$, notons
\begin{equation*}
Y(T)=\Hom(T,\GL(1))
\end{equation*}
le groupe des cocaractères et~$\Lambda=Y(T)\tens\R$ l'espace vectoriel associé. L'\emph{immeuble de Bruhat-Tits de~$G$ sur~$\kk$} est noté~$\Ik(G)$.
C'est le quotient de~$G(\kk) \times \Lambda$ par la relation d'équivalence considérée dans \cite[(1.3.2)]{RTW09} (cf. \cite[2.1]{Tit79}).
Rappelons que~$G^\an$ désigne l'espace analytique associé à~$G$, vu comme espace topologique des semi-normes multiplicatives bornées sur l'algèbre~$\kk[G]$ des fonctions régulières sur~$G$, pour la topologie de la convergence simple. Notons~$\theta:\Ik(G)\rightarrow G^\an$ une application telle que dans l'énoncé~\ref{TheoF3}, et notons~$G_\theta$ son stabilisateur \emph{à droite} dans~$G(\kk)$ qui est compact et ouvert.

Nous fixons un point~$o$ de~$\Ik(G)$, et notons~$G_o$ son stabilisateur dans~$G_\theta$. Le grou\-pe~$G_o$ est compact dans~$G(\kk)$ (\cite[3.2]{Tit79}) et Zariski dense dans~$G$~,\cite[Lemma~1.4]{RTW09}. Nous notons~$H_o$ le groupe compact~$G_o\cap H(\kk)$, et~$\Ik(G)^{H_o}$ le lieu fixe de l'action de~$H_o$ sur~$\Ik(G)$. D'après le Théorème~\ref{TheoF1}, nous pouvons choisir un compact~$C$ de~$\Ik(G)$ tel que~$\Ik(G)^{H_o}=Z_G(H)(\kk)\cdot C$.

La notion de convexité utilisée est celle introduite dans la section~\ref{sectionF}.

\subsubsection*{Définition}\label{Defi Y}
Soit alors~$Y$ le lieu des points~$y$ de~$G(\kk)$ tels que dans l'enveloppe
convexe de~$H_o \cdot y^{-1}\cdot o$, il se trouve un point de~$C$.

On notera que la construction de~$Y$ ne dépend que (de~$G$, de~$H$ et) du choix
de~$\theta$, de~$o$ et de~$C$. La partie~$Y$ ne dépend donc ni de~$\rho$, ni de~$\Omega$. L'énoncé suivant
démontre que la partie~$Y$ est fermée dans~$G(\kk)$ et satisfait la première condition du Théorème~\ref{Theo11}.

\begin{proposition}\label{Prop21}
La partie~$Y$ de~$G(\kk)$ est fermée ; l'intersection~$Y \cap Z_G (H)(\kk)$ est compacte; on a~$G(\kk) = Y \cdot Z_G (H)(\kk)$.
\end{proposition}
\begin{proof}
Montrons que la partie~$Y$ est fermée dans~$G(\kk)$. Par construction,
la partie~$Y$ est invariante à gauche sous le stabilisateur de~$o$. C'est donc une partie stable sous~$G_o$, qui est ouvert. Le complémentaire de~$Y$ est donc ouvert, car
stable sous~$G_o$.

Montrons que le saturé~$Y \cdot Z_G (H)(\kk)$ de~$Y$ par~$Z_G (H)(\kk)$ vaut~$G(\kk)$. Soit~$g$
dans~$G(\kk)$, et formons l'enveloppe convexe~$\langle H_o g^{-1}o\rangle$ de~$H_o g^{
-1}o$ dans~$\Ik(G)$.
L'action de~$H_o$ sur le convexe~$\langle H_o g^{-1}o\rangle$ a au moins un point fixe, d'après~\cite[2.3.1]{Tit79}. Choisissons-en un, disons~$p$. Comme~$p$ est fixe sous~$H_o$, il s'écrit, d'après~\ref{TheoF1}, sous
la forme~$z^{-1}\gamma$ avec~$\gamma$ dans un compact~$C$ et~$z$ dans~$Z_G (H)(\kk)$.
Comme l'action de~$z$ commute à celle de~$H_o$, nous avons~$zH_o g^{
-1}o = H_o zg^{-1}o$. Comme l'action de~$z$ sur~$\Ik(G)$
échange les appartements, et est affine sur chaque appartement, nous avons
\begin{equation*}
\langle H_o zg^{-1}o \rangle=\langle z H_o g^{-1}o \rangle=z\langle H_o g^{-1}o \rangle.
\end{equation*}
Par conséquent~$\langle H_o zg^{-1}o \rangle$ contient le point~$zp = z(z^{-1}\gamma) = \gamma$, qui appartient à~$C$.
Autrement dit~$g z^{-1}$ appartient à~$Y$, ce qu'il fallait démontrer.

Montrons que l'intersection~$Y \cap Z_G (H)(\kk)$ est compacte. Tout d'abord c'est l'intersection de deux fermés, donc c'est un fermé. Comme le stabilisateur~$G_o$ de~$o$ est compact, et que l'action de~$G(\kk)$ sur~$\Ik(G)$ est propre, il suffit de montrer que l'intersection~$(Y^{-1}\cdot o)\cap(Z_G (H)(\kk)\cdot o)$ est relativement compacte. Or~$o$ étant fixe sous~$G_o$, donc sous~$H_o$, l'en\-sem\-ble~$Z_G (H)(\kk) \cdot o$ est contenu dans le lieu fixe~$\Ik(G)^{H_o}$. Mais, par définition même de~$Y$, l'intersection de~$Y^{-1}\cdot o$ avec~$\Ik(G)^{H_o}$ est contenue dans le compact~$C$.
\end{proof}

Ceci étant, il nous reste à démontrer la seconde condition du Théorème~\ref{Theo11}, autrement dit à établir la formule~\eqref{eq2}. Quitte à changer la constante~$c$, la validité
de la formule~\eqref{eq2} ne dépend de la norme~$\Nm{-}$ qu'à équivalence près. Or,~$V$ étant de
dimension finie, toutes les normes (homogènes) sont équivalentes. 
Soit~$B$ une
boule du dual de~$V$. Quitte à appliquer la section C.1, nous pouvons supposer que les hypothèses de la Proposition~\ref{Prop22} concernant la norme~$\Nm{-}$ sont satisfaites.\footnote{\label{pied B}N.d.É.: Cette hypothèse est vraisemblablement superfétatoire, si on considère la boule unité~$B$ du dual après être passé à une extension ultramétrique~$\kk'$ de~$\kk$ à groupe de valuation non discret (dense dans~$\R_>0$). On tâchera de modifier en conséquence la définition de~$c_4$.} 
\begin{proposition}\label{Prop22}
La situation est celle du Théorème~\ref{Theo11}. Nous utilisons les notations
précédentes. En particulier~$\rho$ et~$V$ sont fixés, et nous avons choisi~$\theta$,~$o$ et~$C$ (de sorte, la partie~$Y$ est bien définie).

Supposons que~$\Omega$ soit bornée, et que~$\Nm{-}$ soit~$G_o$-invariante, ultramétrique et\footnote{Voir la note de pied de page précédente.} ne prenne que des valeurs prises par~$\abs{-} : \kk \rightarrow \R$. Alors la formule (2) est satisfaite avec la constante~$c_4/(c_1c_2c_3)$, où
\[
c_1=1+\sup_{f\in C_G(\Ad_\rho)}
	\frac{\pi_\kk(f)}
	{
		\displaystyle{\sup_{\omega\in\Omega}\abs{f}(\omega)}
	},
c_2=\min_{\omega\in\Omega}\NM{\rho(\omega)}^{-1},
c_3=\sup_{f\in C_G(\Ad_\rho)}
	\frac{\abs{f}(\theta(o))}
	{
		\displaystyle{\sup_{k\in K}\abs{f}(k)}
	}
\]
et~$c_4$ sont obtenues en appliquant les Propositions~\ref{Prop32},~\ref{Prop31},~\ref{Prop35}, et~\ref{Prop39} respectivement.
\end{proposition}

\section{Propositions}\label{section3}
Dans cette section nous réunissons quelques arguments généraux qui serviront à la démonstration de la Proposition~\ref{Prop22}. On pourra passer directement à la section suivante et se reporter aux énoncés ci-dessous au besoin. Les notations sont celles introduites dans la section précédente.

\begin{proposition}\label{prop31}\label{Prop31}
Il existe une constante positive inversible~$c_2$ telle que pour tout élément~$g$ de~$G(\kk)$ et tout vecteur~$v$ de~$V$, on ait
\begin{equation}\label{eq3}
\sup_{\omega\in\Omega}\Nm{\rho(g\cdot \omega)(v)}
	\geq
c_2\cdot\sup_{\omega\in\Omega}\Nm{\rho(\omega^{-1}\cdot g\cdot \omega)(v)}.
\end{equation}
\end{proposition}
\begin{proof} Par définition de la norme d'opérateur~$\NM{\rho(\omega)}$, pour tous~$g$,~$\omega$ et~$v$ comme dans l'énoncé, nous avons l'inégalité
\begin{equation}\label{eq4}
\Nm{\rho(\omega^{-1}\cdot g\cdot\omega)}
\geq
\NM{\rho(\omega)}^{-1}\cdot\Nm{\rho(g\cdot\omega)(v)}.
\end{equation}
Posons~$c_2=\inf_{\omega\in\Omega}\Nm{\rho(\omega)}^{-1}$. Comme~$\Omega$ est une partie bornée et non vide, son image par~$\omega\mapsto\NM{\rho(\omega)}^{-1}$ est une partie bornée et non vide de~$\R_{<0}$. La constante~$c_2$ est donc positive et inversible, et répond à l'énoncé.
\end{proof}
\begin{proposition}\label{prop32}\label{Prop32}
Il existe une constante positive inversible~$c_1$ telle que pour tout coefficient matriciel~$f$ dans~$C_H(\Ad_\rho)$, on ait
\begin{equation}\label{eq5}
\sup_{\omega\in\Omega}\abs{f(\omega)}\geq\frac{1}{c_1}\abs{\pi_\lie{\kk}(f)}.
\end{equation}
\end{proposition}
\begin{proof} Comme~$\Omega$ est borné, l'application~$\Nm{-}_\Omega:f\mapsto \sup_{\omega\in\Omega}\abs{f}(\omega)$ est bien définie sur~$C_H(\Ad_\rho)$. C'est manifestement une semi-norme. D'eprès la condition~\eqref{*}, elle ne s'annule pas: c'est une norme. Comme~$C_H(\Ad_\rho)$ est de dimension finie, l'application linéaire~$\pi_\kk$, de~$C_H(\Ad_\rho)$ sur~$\kk$, est un opérateur borné, relativement à~$\Nm{-}_\Omega$ et~$\abs{-}$. Si~$\NM{\pi_\kk}$ désigne sa norme en tant qu’opérateur~$(C_H(\Ad_\rho),\Nm{-}_\Omega)\to (\kk,\abs{-})$, alors
\begin{equation}\label{eq6}
\NM{\pi_\kk} \cdot \sup_{\omega\in\Omega}\abs{f(\omega)} \geq \abs{\pi_\kk(f)}.
\end{equation}
Par conséquent~$c_1=\NM{\pi_\kk}$ convient si~$\NM{\pi_\kk}\neq0$. Si\footnote{En définitive le cas~$\pi_\kk=0$ ne se produit pas si~$V$ est non nulle. Voir~p.\,\pageref{constantes coeffs} pourquoi les constantes non nulles sont coefficients matriciels dans~$C_H(\Ad_\rho)$.}, en revanche~$\NM{\pi_\kk}=0$, alors~$c_1$ convient. Quoiqu'il en soit,~$c_1=1+\NM{\pi_\kk}$ convient toujours.
\end{proof}
\begin{proposition}\label{prop33}\label{Prop33} Pour tout vecteur~$v$ de~$V$, pour toute forme linéaire~$\phi$ dans~$V^\vee$ et tout élément~$g$ de~$G(\kk)$,
\begin{equation}\label{eq7}
\omega\mapsto(\rho(\omega^{-1}\cdot y\cdot \omega)(v)|\phi)
\end{equation}
définit une fonction sur~$H$ (resp.~$G$) appartenant à~$C_H (\Ad_\rho )$ (resp.~$C_G (\Ad_\rho )$).
\end{proposition}
\begin{proof} Comme~$e\mapsto(e(v)|\phi)$ est une forme linéaire sur~$\gl(V)$, l'application~\eqref{eq7} est un coefficient matriciel de l’action~$\Ad_\rho$.
\end{proof}
\begin{proposition}\label{propr34}\label{Prop34}
Nous utilisons les hypothèses du Théorème~\ref{Theo11} concernant la représentation~$\rho$. Pour toute forme~$\kk$-linéaire~$\phi$ dans~$V^\vee$, la fonction constante sur~$H$
\begin{equation}\label{eq8}
\pi_\kk\left(\omega\mapsto\left(\rho\left(\omega^{-1}\cdot y\cdot\omega^{\vphantom{l}}\right)^{\vphantom{l}}(v)\middle|\phi\right)\right)
\end{equation}
vaut
\begin{equation}\label{eq9}
\left(\pi_{\z}\left(\rho(y)^{\vphantom{l}}\right)^{\vphantom{i}}(v)\middle| \phi \right)
\end{equation}
et, lorsque~$y$ varie dans~$G$, définit une fonction régulière appartenant à~$C_G(\Ad_\rho)$ et invariante sous l'action par conjugaison de~$H$ sur~$G$.
\end{proposition}
\begin{proof}
La fonction~$y\mapsto\left(\pi_{\z}(\rho(y))(v)\middle|\pi\right)$ est manifestement régulière, et est invariante pour l'action par conjugaison de~$H$ sur~$G$, vu que, pour~$\omega$ dans~$H$,
\begin{equation}\label{eq10}
	\left(\pi_{\z}(\rho(\omega y \omega^{-1})\right)
		=
	\rho(\omega)\pi_{\z}(\rho(y))\rho(\omega^{-1})
		=
	\pi_{\z}(\rho(y))
\end{equation}
car~$\pi_\z$ est~$H$-équivariant et d'image dans~$\z$.

Pour établir l'égalité de~\eqref{eq8} et~\eqref{eq9}, considérons l'application~$\Phi:\gl(V)\to C_H(\Ad_\rho)$  qui envoie~$e$ vers le coefficient matriciel~$\omega\mapsto\left(\rho(\omega)e(\rho(y))\rho(\omega^{-1})\middle|\phi\right)$.
C'est une application~$H$-équivariante, et, d'après la Remarque~\emph{i)} en page~\pageref{Remarquei}, elle commute aux projecteurs~$\pi_\z$ et~$\pi_\kk$. Autrement dit
\[
\pi_\kk(\omega\mapsto(\rho(\omega)e\rho(\omega^{-1})(v)|\phi)=\omega\mapsto(\rho(\omega)\pi_\z(e)\rho(\omega^{-1})(v)|\phi).
\]
Prenons~$e=\rho(y)$. Alors le membre de gauche s'identifie à~\eqref{eq8}, et, d'après~\eqref{eq10}, le membre de droite s'identifie à~\eqref{eq9}.
\end{proof}
\begin{proposition}\label{prop35}\label{Prop35}
Pour tout point~$o$ de~$\Ik(G)$, l'application~$f\mapsto \sup_{k\in G_o}\abs{f}(k)$ et la restriction de la semi-norme~$\theta(o)$ définissent, sur~$C_G(\Ad_\rho)$, deux normes comparables.

En particulier, il existe une constante positive et inversible~$c_3$ telle que, pour toute fonction~$f$ dans~$C_G(\Ad_\rho)$, on a~$\sup_{k\in G_o}\abs{f}(k)\geq\frac{1}{c_3}\abs{f}(\theta(o))$.
\end{proposition}
\begin{proof}
L’existence de~$c_3$ découle de la définition de la comparabilité des nor\-mes de la section~\ref{sectionC1}. D’après cette section, il suffit de vérifier que l’on a
bien deux normes~$(k,\abs{-})$-homogènes.

L’application~$f \mapsto \sup_{k\in G_o}\abs{f}(k)$ est bien définie car~$G_o$ est compact. C’est manifestement
une semi-norme~$(\kk,\abs{-})$-homogène. Comme~$G_o$ est ouvert, donc Zariski dense dans~$G$,
cette semi-norme ne s’annule en aucune fonction régulière. C’est donc une norme.

Quant à~$\theta(o)$, comme il s’agit par définition d’une semi-norme (non nulle,
\cite[1.1]{Ber90}) multiplicative~$(\kk,\abs{-})$-homogène sur~$\kk[G]$, il suffit de vérifier que
c’est en fait une norme. Comme~$\theta(o)$ est multiplicative, son noyau définit un idéal
de~$\kk[G]$. Comme~$\theta(o)$ vaut~$1$ en~$1$, c’est idéal est strict. Or le stabilisateur de~$p$
dans~$G(\kk)$ est Zariski dense. Cet idéal définit une sous-variété~$G(\kk)$-invariante
de~$G$: cette sous-variété ou bien est vide ou bien vaut~$G$ lui-même. Or~$G$ est réduit,
et l’idéal considéré ne contient pas l’unité. Par conséquent cet idéal est nul:~$\theta(p)$
ne s’annule pas sur~$\kk[G]$ et \emph{a fortiori} sur~$C_G(\Ad_\rho)$.
\end{proof}
L'énoncé suivant est un corollaire à l’énoncé~\ref{TheoF3}. La notion de convexité est précisée dans la section~\ref{sectionF} correspondante, à laquelle nous renvoyons. Mentionnons juste que cette notion de convexité est naturellement induite par la structure affine par morceaux standard sur l'immeuble~$\Ik(G)$.
\begin{proposition}[Convexité]\label{prop36}\label{Prop36}
Pour toute fonction régulière~$f$ sur~$G$,
l'application~$p \mapsto \abs{f}(\theta(p))$ est convexe sur $\Ik (G)$.
\end{proposition}
Il s'agit du Corollaire~\ref{coroF4} au Théorème~\ref{TheoF3}.
\begin{proposition}[Hyperbolicité]\label{prop37}\label{Prop37}
Pour tout point~$p$ de~$\Ik (G)$, l'enveloppe
convexe de~$H_o \cdot p$ contient un point fixe de~$H_o$.
\end{proposition}
\begin{proof} Notons que comme~$\Ik(G)$ est localement réunion finie d'appartements, et que~$H_o\cdot p$ est compact (C'est l'image du groupe compact~$H_o$ par une application continue vers un espace séparé), l'enveloppe convexe de~$H_o \cdot p$ est compacte. Il suffit alors d'appliquer~\cite[2.3.1]{Tit79} et~\cite[3.2.3]{BT72}.
\end{proof}
\begin{proposition}\label{Prop38}Soient~$Y$ et~$C$ comme en page~\pageref{Defi Y}.
Soit~$f$~une fonction convexe sur~$\Ik (G)$ et~$H_o$-invariante à gauche sur~$Y \cdot o$, et un point~$p$ appartenant à~$Y \cdot o$ . Alors
\begin{equation}\label{eq11}
f(p)\geq\inf_{\gamma\in C }f(\gamma).
\end{equation}
\end{proposition}
\begin{proof}Comme~$f$ est~$H_o$-invariante sur~$Y\cdot o$, on a~$f(p)=\sup_{x\in H_o\cdot p} f(x)$. Notons~$\langle H_o\cdot p\rangle$ l'enveloppe convexe de~$H_o\cdot p$. Comme~$f$ est convexe,~$\sup_{x\in H_o\cdot p} f(x)=\sup_{x\in \langle H_o\cdot p\rangle} f(x)$. Lorsque~$p$ appartient à~$Y\cdot o$, l'intersection~$C\cap\langle H_o\cdot p\rangle$ est non vide. D'où
\[
f(p)=\sup_{x\in\langle H_o\cdot p\rangle} f(x)\geq \sup_{x\in \langle H_o\cdot p\rangle\cap C} f(x)\geq\inf_{x\in \langle H_o\cdot p\rangle\cap C} f(x)\geq \inf_{\gamma\in C} f(\gamma).
\]
\end{proof}
Dans la proposition suivante,~$C$ désigne le compact de~$\Ik(G)$ défini dans la section~\ref{section2} précédente, et~$B$ la boule unité du dual de~$V$ (cf. la note de l'éd. au pied de la page~\pageref{pied B}).

Dans cette proposition, on étudie les coefficients matriciels qui sont de la forme~$g\mapsto\phi(\pi_\z(\rho(g))(v)$ comme fonction sur~$G^\an$, et en particulier sur l'image de~$C$ dans~$G^\an$ par l'application~$\theta:\Ik(G)\to G^\an$. On notera donc
\[
\abs{\phi}(\pi_\z(\rho(\theta(\gamma)))(v))
\]
a valeur obtenue en appliquant, pour~$\gamma$ dans~$C$ la norme~$\theta(\gamma)$ au coefficient matriciel~$g\mapsto\phi(\pi_\z(\rho(g))(v))$. Pour alléger les notations, on pourra omettre~$\theta$ et~$\rho$ dans les notations, soit
\[
\abs{\phi}(\pi_\z(\gamma)(v))
=
\abs{\phi}(\pi_\z(\rho(\theta(\gamma)))(v)).
\]
\begin{proposition}\label{Prop39}
Il existe une constante~$c_4$ telle que pour tout~$\gamma$ de~$C$, et tout vecteur~$v$ dans~$V$, on a
\begin{equation}\label{eq12}
\sup_{\phi\in B}\pi_{\z}(\rho(\theta(\gamma)))(v)\geq\Nm{v}/c_4.
\end{equation}
\end{proposition}
\begin{proof} Remarquons tout d'abord que, d'après la Proposition~\ref{prop35}, la semi-nor\-me~$\theta(\gamma)$ est une norme, pour tout~$\gamma$ dans~$C$.

La restriction de la norme~$\theta(\gamma)$ sur~$\kk[G]$ à~$C_G(\Ad_\rho)$ dépend continûment de~$\gamma$, pour la
topologie faible. Comme~$C_G(\Ad_\rho)$ est de dimension finie et~$C$ est compact, ces nor\-mes sont « uniformément équivalentes » pour~$\gamma$ dans~$C$:  il existe une constante~$c'$ telle que pour tous~$\gamma$ et~$\gamma'$ dans~$C$ nous ayons
\[\forall f\in C_G(\Ad_\rho),\abs{f}(\theta(\gamma))\geq c'\abs{f}(\theta(\gamma')).\] Il suffira de vérifier la formule~\eqref{eq12} pour une constante~$c_4'$ et \emph{pour un seul~$\gamma$ de~$C$}: la constante~$c_4=c'c_4'$ conviendra alors pour tout~$\gamma$ dans~$C$.

Fixons ~$\gamma$ dans~$C$. La formule est évidente pour~$v = 0$. Nous pouvons donc supposer que~$v \neq 0$, et même, par homogénéité, que, pour une certaine constante~$c''$ ne dépendant que de~$\Nm{-}$, le vecteur~$v$ appartienne au compact~$V_{c''}$ où l'inégalité~$1/c'' \leq \Nm{v} \leq c''$ est satisfaite.

Tout revient ainsi à montrer que
\[
\inf_{v\in V_{c''}}\sup_{\phi\in B}(\pi_\z(\gamma)(v))>0.
\]
Soit, par l'absurde, un suite~$(v_n)_{n\in \Z_{\geq0}}$ de~$V_{c''}$ telle que
\[
\lim_{n\in \Z_{\geq0}}\sup_{\phi\in B}(\pi_\z(\gamma)(v_n))>0.
\]
Comme~$V_{c''}$ est compact, et n'adhère pas à~$0$, la suite~$v_n$ a une valeur d'adhérence non nulle~$v_\infty$. Nous allons montrer, ce qui sera une contradiction, que~$v_\infty$ est nécessairement nul.

Pour tout~$\phi_0$ dans~$B$, on conclut de l'encadrement
\[
0\leq\abs{\phi_0}\left(\pi_\z(\gamma)(v_n)\right)\leq\sup_{\phi\in B}\abs{\phi}\left(\pi_\z(\gamma)(v_n)\right)\to 0
\]
que~$\lim_{n\in \Z_{\geq0}}\abs{\phi_0}\left(\pi_\z(\gamma)(v_n)\right)=0$. Par continuité de~$v\mapsto\abs{\phi_0}\left(\pi_\z(\gamma)(v)\right)$, il s'ensuit que l’on a~$\abs{\phi_0}(\pi_\z(\gamma)(v_\infty)=0$.

Ainsi, pour tout~$\phi$ de~$B$, on a~$\abs{\phi}(\pi_\z(\gamma)(v_\infty))=0$. Autrement dit, comme~$\theta(\gamma)$ est une
norme, chaque coefficient matriciel~$g\mapsto\pi(\pi_\z(\rho(g))(v_\infty)$ 
est identiquement nul. Par conséquent, pour tout~$g$ dans~$G(\kk)$, 
le vecteur~$\pi_\z (\rho(g))(v_\infty)$ est nul. Mais,
lorsque~$g$ vaut l'élément neutre,
\[
\pi_\z(\rho(g))(v_\infty)=\pi_\z(\Id_V)(v_\infty)=\Id_V(v_\infty)=v_\infty\neq0.
\]

\end{proof}
\section{Démonstration}\label{section4}
Démontrons la Proposition~\ref{Prop22}. Nous utilisons les notations de la section~\ref{section2}, et
les arguments de la section~\ref{section3}.
\begin{proof}
Comme la norme~$\Nm{-}$ est supposée~$G_o$-invariante, nous pouvons sub\-sti\-tuer~$\sup_{k\in G_o}\Nm{\rho(k\cdot y\cdot \omega)}$ à~$\rho(y\cdot \omega)$ dans la formule~\eqref{eq2}, ce qui donne
\begin{equation}\label{eq13}
\forall y\in Y, \forall v\in V, \sup_{\omega\in\Omega}\sup_{k\in G_o}\Nm{\rho(k\cdot y\cdot \omega)}\geq \Nm{v}/c.
\end{equation}
D'après la Proposition~\ref{Prop31}, il suffit d'établir
\begin{equation}\label{eq14}
\forall y\in Y, \forall v\in V, \sup_{\omega\in\Omega}\sup_{k\in G_o}\Nm{\rho(\omega^{-1}\cdot k\cdot y\cdot \omega)}\geq\frac{1}{c\cdot c_2} \Nm{v}.
\end{equation}
Soit~$V^\vee$ le dual algébrique de~$V$, et notons~$B$ sa boule unité.
D'après les hypothèses sur~$\Nm{-}$ 
nous avons~$\Nm{v}=\sup_{\phi\in B}\abs{\phi(v)}$. La formule qui précède équivaut donc à la suivante.
\begin{equation}\label{eq15}
\forall y\in Y, \forall v\in V, \sup_{\omega\in\Omega}\sup_{\phi\in B}\sup_{k\in G_o}\abs{\phi}\left(\rho(\omega^{-1}\cdot k\cdot y\cdot \omega)\right)\geq\frac{1}{c\cdot c_2} \Nm{v}.
\end{equation}
D'après la Proposition~\ref{Prop33}, la fonction~$\omega\mapsto {\phi}\left(\rho(\omega^{-1}\cdot k\cdot y\cdot \omega)\right)$ appartient à~$C_H(\Ad_\rho)$. Appliquant la Proposition~\ref{Prop32}, il sort
\begin{equation}\label{eq16}
 \sup_{\omega\in\Omega}\abs{\phi}\left(\rho(\omega^{-1}\cdot k\cdot y\cdot \omega)\right)
 	\geq
\frac{1}{c_1}
	\abs{
		\pi_{\kk}
		\left(
			\omega
				\mapsto
			{\phi}\left(\rho(\omega^{-1}\cdot k\cdot y\cdot \omega)\right)
		\right)
		}.
\end{equation}
D'après la Proposition~\ref{Prop34}, le membre de droite de~\eqref{eq16} vaut
\begin{equation}\label{eq17}
\frac{1}{c_1}\abs{\phi}(\pi_{\z}(k\cdot y)(v)).
\end{equation}
D'après la Proposition~\ref{prop35}, il existe une constante positive et inversible~$c_3$ telle que
\begin{equation}\label{eq18}
\sup_{k\in G_o}\abs{\phi}(\pi_{\z}(k\cdot y)(v))\geq\frac{1}{c_3}\abs{\phi}(\pi_{\z}(\theta(o)\cdot y)(v)).
\end{equation}
Par conséquent, nous avons établi, combinant~\eqref{eq16}, \eqref{eq17} et~\eqref{eq18},
\begin{equation*}
\forall y\in Y, \forall v\in V, 
\sup_{\omega\in\Omega}\sup_{k\in G_o}\abs{\phi}\left(\rho(\omega^{-1}\cdot k\cdot y\cdot \omega)\right)
\geq
\frac{1}{c_1 c_3}\sup_{\phi\in B}\abs{\phi}(\pi_{\z}(\theta(o)y)(v)),
\end{equation*}
d'où, considérant la borne supérieure relative aux~$\phi$ dans~$B$,
\begin{equation}\label{eq19}
\forall y\in Y, \forall v\in V, 
\sup_{\omega\in\Omega}\sup_{\phi\in B}\sup_{k\in G_o}\abs{\phi}\left(\rho(\omega^{-1}\cdot k\cdot y\cdot \omega)\right)
\geq
\frac{1}{c_1 c_3}\sup_{\phi\in B}\abs{\phi}(\pi_{\z}(\theta(o)y)(v)).
\end{equation}
Ainsi, pour démontrer~\eqref{eq15}, il suffit d'établir
\begin{equation}\label{eq20}
\frac{1}{c_1 c_3}\sup_{\phi\in B}\abs{\phi}(\pi_{\z}(\theta(o)y)(v))
\geq
\Nm{v}\frac{1}{c\cdot c_2}.
\end{equation}
D'après la Proposition~\ref{Prop34}, la fonction~$g\mapsto \pi(\pi_{\z}(g)(v))$ est régulière sur~$G$ en la variable~$g$, invariante sous l'action de~$H$ par conjugaison. Or, pour~$h$ dans~$H_o$, et~$y^{-1}\cdot o$ dans~$Y^{-1}\cdot o$, nous avons
\begin{equation*}
h\theta(y^{-1}o)h^{-1} = h\theta(y^{-1}o)
\end{equation*}
car~$G_\theta$ contient~$h^{-1}$, et
\begin{equation*}
h\theta(y^{-1}o) = \theta(h y^{-1}o)
\end{equation*}
car~$\theta$ est équivariante. Sur~$Y^{-1}\cdot o$, la fonction~$y^{-1}o\mapsto\abs{\phi}(\pi_{\z}(\theta(y^{-1}o))(v))$ est donc invariante à gauche sous~$H_o$.

D'après la Proposition~\ref{Prop36}, la fonction~$p\mapsto\abs{\phi}(\pi_{\z}(\theta(p))(v))$ est convexe sur~$\Ik(G)$. Par conséquent la fonction~$g\mapsto \sup_{\phi\in B}\abs{\phi}(\pi_{\z}(g)(v))$
est convexe sur~$\Ik(G)$, et sa restriction à~$Y^{-1}\cdot o$ est invariante à gauche sous~$H_o$. D'après la Proposition~\ref{Prop38}, on a, pour tout~$y$ de~$Y$, 
\begin{equation*}
\sup_{\phi\in B} \abs{\phi}(\pi_{\z}(y^{-1}o)(v))
\geq
\inf_{\gamma\in C}
\sup_{\phi\in B}\abs{\phi}(\pi_{\z}(\gamma)(v)).
\end{equation*}
Or, d'après la Proposition~\ref{Prop39}, nous avons
\begin{equation*}
\forall v\in V, \forall\gamma\in C, 
\sup_{\phi\in B}\abs{\phi}(\pi_{\z}(\gamma)(v))\geq \Nm{v}/c_4.
\end{equation*}
Ce qui démontre bien la formule~\eqref{eq20}, avec la constante~$c=
\frac{c_1c_2c_3}{c_4}$.
\end{proof}

\appendix
\renewcommand{\appendixpagename}{Annexe}
\appendixpage
Dans cette annexe, nous faisons quelques rappels sur  les espaces analytiques et les immeubles. Nous y démontrons notamment (Thé\-o\-rè\-me~\ref{TheoF1})
un résultat de décomposition sur les immeubles et (Théorème~\ref{TheoF3}) un résultat sur la convexité logarithmique des fonctions régulières sur l'immeuble, une fois plongé dans l'espace analytique.

\section{Normes ultramétriques}\label{sectionC}\label{sectionC1}
Soit~$V$ un espace vectoriel sur un corps ultramétrique~$\kk$ de valeur absolue notée~$\abs{-}$. La to\-pologie de~$V$, topologie produit relative à une base de~$V$, est intrinsèque: les automorphismes de changement de base de~$\kk^{\dim(V)}$ sont des homéomorphismes. Nous appelons~\emph{norme} sur~$V$ une application~$V\to\R$ telle que l'application~$(x,y)\mapsto \Nm{x-y}$ définisse une distance compatible à la topologie. Cette norme est dite~\emph{$(\kk,\abs{-})$-homogène}, ou simplement~\emph{homogène}, lorsque toute homothétie de facteur~$\lambda$ agit sur les distances d'un facteur~$\abs{\lambda}$:
\addtocounter{equation}{1}
\begin{equation}\label{eq22}
\text{
pour tout~$\lambda$ dans~$\kk$ et~$v$ dans~$V$, nous avons~$\Nm{\lambda\cdot v}=\abs{\lambda}\cdot\Nm{v}$.}
\end{equation}
On dit que la norme~$\Nm{-}$ est \emph{ultramétrique} si
\begin{equation}\label{eq23}
\forall x,y\in V, \Nm{x+y}\leq\max\left\{\Nm{x\vphantom{y}};\Nm{y}\right\}.
\end{equation}

Nous dirons que deux normes~$\Nm{-}$ et~$\Nm{-}'$ sont \emph{comparables}, ou~\emph{équivalentes}, s'il existe une constante positive et inversible~$C$ telle que pour tout~$v$ dans~$V$ , nous
ay\-ons~$\Nm{v}\leq C\Nm{v}'$ et~$\Nm{v}'\leq C\Nm{v}$.

Remarquons que~$\abs{-}$ et~$\abs{-}^2$ sont deux normes sur~$\kk$ qui ne sont pas comparables, sauf si~$\kk$ est discret. \emph{A contrario} deux normes homogènes sur~$V$ sont toujours comparables. Nous ne l'appliquerons qu'à des corps~$\kk$ localement compacts, auquel cas c'est immédiat. Voir~\cite{TheseV} pour le cas général.

\emph{Dorénavant les normes seront supposées homogènes.}

Étant donné une norme sur~$V$, son dual algébrique acquiert une norme duale: à toute forme~$\kk$-linéaire sur~$V$ on associe  sa  norme
 en tant qu'application linéaire de~$V$ vers~$\kk$. Autrement dit on pose
$\Nm{\phi}=\NM{\phi}=\sup_{\Nm{v}\leq 1}\Nm{\phi(v)}$.

\begin{lemme}\label{LemmeC1} Soit~$V$ un espace vectoriel normé de dimension finie non nulle sur~$\kk$, et doit~$B$ la boule unité de son dual. Alors on a l'inégalité
\[
\forall v\in V, \Nm{v}\geq\sup\abs{\phi}(v),
\]
et il n'y a égalité, simultanément pour tout~$v$ de~$V$, que si et seulement si~$\Nm{-}$ est ultramétrique et si les valeurs, dans~$\R$, prises par~$\Nm{-}$ sont celles prises par~$\abs{-}$.
\end{lemme}
\begin{proof}
L'inégalité résulte de ce que pour tout~$\phi$ dans~$B$, nous avons~$\Nm{\phi}\leq 1$ et de ce que, par définition de la norme triple,~$\Nm{\phi(v)}\leq\NM{\phi}\cdot\Nm{v}$.

Le sens direct de l'équivalence découle de~\cite[II §1, Prop. 4 (p. 26)]{Wei74}.

Dans le sens réciproque, on vérifie directement que le membre de droite est une norme ultramétrique et ne prend que des valeurs prises par~$\abs{-}$.
\end{proof}

\section{Espaces analytiques d'après Berkovich}\label{sectionD}
Le but de cette section est tout d'abord de rappeler la construction des espaces analytiques de Berkovich, et de l'espace analytifié d'une variété algébrique. La référence exhaustive standard est~\cite{Ber90}.
L'autre but est d'étendre la propriété de convexité du polygone de Newton au fonctions analytiques sur restreinte à l'appartement d'un tore~$T$, une fois plongé dans l'espace analytique de~$T$ ou d'un groupe algébrique~$G$ contenant~$T$.
\subsection{Semi-normes multiplicatives homogènes}\label{sectionD1}
Soit~$\kk$ un corps local muni
d'une valeur absolue ultramétrique~$\abs{-}$. Sur une~$\kk$-algèbre commutative unifère~$A$, on appelle \emph{semi-norme multiplicative} une application\emph{non constante}~$\Nm{-} : A \xrightarrow{} \R_{\geq 0}$ 
qui soit multiplicative, c.-à-d. telle que
\begin{equation}\label{eq24}
\forall f,g\in A,~\Nm{fg}=\Nm{f}\cdot\Nm{g}
\end{equation}
et vérifiant l'inégalité triangulaire
\begin{equation}\label{eq25}
\forall f,g\in A,~\Nm{f+g}\leq\Nm{f}+\Nm{g}.
\end{equation}
Une telle application envoie l'unité~$1_A$ sur~$1$ et l'élément nul~$0_A$ sur~$0$. Elle est dite~\emph{$(k,\abs{-})$-homogène}, ou simplement \emph{homogène} lorsque
\begin{equation}\label{eq26}
\forall\lambda\in\kk, \forall f\in A,~\Nm{\lambda\cdot f}\leq \abs{f}\cdot\Nm{f}.
\end{equation}
Par multiplicativité, il revient au même d'imposer que~$\Nm{-}$ étende~$\abs{-}$, au sens où l’on a~$\Nm{\lambda\cdot 1_A} = \abs{\lambda}$.

\subsection{Analytification}\label{sectionD2}
Soit~$V$ une variété algébrique affine sur~$\kk$. Suivant
V. Berkovich, \cite[1.5.1]{Ber90}, nous appellerons \emph{espace analytique associé à}~$V$ l'espace topologique~$V^\an$ formé de l'ensemble des semi-normes multiplicatives
homogènes sur l'algèbre~$\kk[V ]$ des fonctions régulières sur~$V$, pour la topologie de la convergence simple.
Ce sont aussi les espaces topologiques sous-jacents à certains espaces analytiques que V. Berkovich définit en~\cite[3.1]{Ber90} (cf.~\cite[3.4.2]{Ber90}). Pour toute fonction régulière~$f$ sur~$V$, nous notons~$\abs{f}$ a fonction réelle~$x\mapsto x(f)$ sur~$V^\an$.  Par définition, la topologie sur~$V^\an$ est la plus grossière pour laquelle, pour toute fonction régulière~$f$, la fonction~$\abs{f}:V^\an\rightarrow \R_{\geq0}$ est continue.

La construction de l'espace analytique associé est fonctorielle, et covariante,
de la catégorie~$\Aff_\kk$ des variétés affines sur~$\kk$ dans celle des espaces topologiques. En effet tout
morphisme~$A \to A'$ d'algèbres permet, par composition, de produire une seminorme multiplicative~$A \rightarrow \R_{\geq0}$ à partir d'une semi-norme multiplicative homogène sur~$A' \rightarrow \R_{\geq0}$. Cette opération est bien sûr compatible à la convergence simple.
Pour tout morphisme~$\Phi:V\rightarrow V'$ entre variétés algébriques, nous notons~$\Phi^\an$ l'application continue correspondante~$V^\an\rightarrow V'^\an$.

\subsection{Des tores déployés analytiques \ldots}\label{sectionD3}

Soit~$T$ un tore déployé sur~$\kk$ et notons~$X(T)=\Hom(T,\GL(1))$ son groupe des caractères. On identifie l'algèbre~$\kk[T ]$ des fonctions régulières sur~$T$ à l'algèbre de groupe~$\kk[X(T )]$. Tout
caractère~$\chi: T \rightarrow\GL(1)$, définit, par composition une application additive
\[
Y(T)=\Hom(\GL(1),T)\xrightarrow{}\Hom(\GL(1),\GL(1))\simeq\Z.
\]
Par linéarité, chaque caractère définit une forme linéaire sur l'espace vectoriel réel~$\Lambda=Y(T)\tens\R$ que l'on notera~$\lambda\mapsto\langle\chi,\lambda\rangle$. Pour tout~$\lambda$ dans~$\Lambda$, l'application qui envoie une fonction régulière~$f=\sum_{\chi\in X(T)}a_\chi\cdot\chi$ dans~$\kk[T]$ sur
\begin{equation}\label{eq27}
\max_{\chi\in X(T)}\abs{a_\chi}\cdot\abs{\varpi}^{\langle\chi,\lambda\rangle},
\end{equation}
(où~$\abs{\varpi}$ est la valeur absolue d'une uniformisante~$\varpi$ de~$\kk$) définit clairement une norme homogène sur~$\kk[T]$; cette norme est multiplicative d'après~\cite[2.1, p.~21]{Ber90}.  La formule~\eqref{eq27} induit donc une application, que nous noterons~$\lambda \mapsto \theta_T (\lambda)$, de~$\Lambda$ dans~$T^\an$.

Lorsque l'on fixe~$f=\sum_{\chi\in X(T)}a_\chi\cdot\chi$ dans~$\kk[T]$, et 
que l'on fait varier le paramètre~$\lambda$, le logarithme
\[
\log\left(\abs{f}(\theta_T(\lambda))\right)
=\max_{\chi\in X(T)}
\left(
\log\abs{a_\chi}+\log\abs{\varpi}\cdot\langle\chi,\lambda\rangle
\right)
\]
est le maximum d'un nombre fini de fonctions affines en~$\lambda$. C'est donc une fonction convexe sur~$\Lambda$, et c'est en particulier une fonction continue. Pour toute fonction régulière~$f$ sur~$T$, la composition de~$\abs{f}$ avec~$\theta_T$, qui est logarithmiquement
convexe, est continue. Donc l'application~$\theta_T$ est continue par définition de la topologie sur~$T^\an$.

\subsubsection{}\label{sectionD31}
Le groupe~$T (\kk)$ agit sur~$T$ par translation. Il agit par transport de structure sur~$\kk[T ]$ puis sur~$T^\an$. Concrètement, un élément~$\mu$ de~$T (\kk)$ agit sur~$\kk[T ]$ en envoyant la fonction régulière
\[
x\mapsto\sum_{\chi\in X(T)} a_\chi\cdot \chi(x)
\]
sur la fonction régulière
\[
x\mapsto
\sum_{\chi\in X(T)} a_\chi\cdot \chi(x\mu^{-1})
=
\sum_{\chi\in X(T)} a_\chi\cdot \chi(\mu^{-1})\chi(x).
\]
Soit~$\lambda(\mu)$ l'élément de~$\Lambda$ tel que l'on ait
\[
\langle\chi,\lambda(\mu)\rangle=\log_{\abs{\varpi}}(\chi(\mu))
\]
pour tout~$\chi$ de~$X(T)$. Alors l'action de~$\mu$ envoie la norme~$\theta_T (\lambda)$ sur~$\theta_T (\lambda-\lambda(\mu))$.

\subsection{\ldots aux groupes algébriques affines.}\label{sectionD4}

Soit~$\Phi: T \to G$ un morphisme de
variétés algébriques affines du tore T dans un groupe algébrique affine G. Composant~$\theta_T : \Lambda \rightarrow T^\an$ par $\Phi^\an: T^\an \rightarrow G^\an$, nous obtenons une application continue
de~$\Lambda$ dans~$G^\an$. En outre, pour toute fonction régulière~$f$ sur~$G$, la fonction
réelle~$\abs{f}\circ\Phi^\an \circ\theta_T$ sur~$\Lambda$ s'identifie à~$\abs{f \circ\Phi}\circ\theta_T$, et est par conséquent logarithmiquement
convexe.

Notons que l'action à droite du groupe~$G(\kk)$ sur la variété~$G$ induit par une fonctorialité une action de~$G(\kk)$ sur~$G^\an$. En particulier nous pouvons définir, pour tout~$g$ dans~$G(\kk)$, l'application $(\Phi^\an \circ \theta_T )\cdot g$ translatée de~$\Phi^\an \circ\theta_T$ par~$g$.

\emph{Nous notons~$\kk_s$ une extension algébrique séparablement close de~$\kk$.}
\subsubsection{}\label{sectionD41}
Nous allons généraliser cette construction aux éléments de~$G(\kk_s)$.
Pour toute extension séparable finie~$\kk'$ de~$\kk$, on a un tore déployé~$T_{\kk'}=T\tens_\kk\kk'$ sur~$\kk'$, d'algèbre~$\kk'[T_{\kk'}]=\kk[T ] \tens_\kk\kk'$ isomorphe à~$\kk'[X(T)]$. Par restriction des normes de~$\kk'[T_{\kk'}]$ à~$\kk[T]$ on construit une application~${T_{\kk'}}^\an\to T^\an$. La formule~\eqref{eq27} s'étend à~$\kk[T_{\kk'}]$ et définit encore une norme~\emph{multiplicative} (\cite[2.1 (p.~21)]{Ber90}) homogène.
Nous obtenons ainsi une application~$\theta_{T_{\kk'}}:\Lambda\to {T_{\kk'}}^\an$ dont la composée avec l'application~${T_{\kk'}}^\an\to T^\an$ redonne l'application~$\theta_T$. De surcroît cette extension est compatibles aux morphismes d'extensions~$\kk\to\kk'\to\kk''$.

Nous en tirons une conséquence. Si~$g$ est un élément dans~$G(\kk')$, faisons agir~$g$ à droite sur le foncteur des points de~$G$ restreint à la catégorie des~$\kk'$-agèbres: pour une~$\kk'$-algèbre~$A$, l'élément~$g$ agit par translation à droite sur le groupe~$G(A)$, \emph{via} son image par~$G(\kk')\to G(A)$. Il correspond une action, disons~$a_g$, de~$g$ sur~$\kk'[G]$. D'où, par composition, un morphisme~$\kk[G]\to\kk'[G]\xrightarrow{a_g}\kk'[G]\xrightarrow{\Psi_{\kk'}}\kk'[T]$, où~$\Psi_{\kk'}$ est le morphisme de~$\kk'$-algèbres déduit de~$\Phi$. Ainsi, pour chaque~$\lambda$ dans~$\Lambda$, de la norme multiplicative homogène correspondante~$\kk[T]\to\R_{\geq0}$, par~\eqref{eq27}, on déduit, par composition, une norme multiplicative homogène~$\kk[G]\to\R_{\geq0}$.
Nous la noterons~$(\Phi^\an\circ\theta_T)g$ l'application de~$\Lambda$ dans~$G^\an$ ainsi obtenue.

Cette construction est manifestement compatible aux extensions: si~$\kk'$ est une extension finie de~$\kk'$ et~$g''$ l'image de~$g$ dans~$G(\kk'')$, alors~$(\Phi^\an\circ\theta_T)g=(\Phi^\an\circ\theta_T)g''$. Elle ne dépend donc que de l'image de~$g$ dans~$G(\kk_s)$, peu importe le morphisme~$\kk'\to\kk_s$. Autrement dit cette construction ne dépend que de l'image de~$g$ par~$G(\kk_s)\to G^\an$.

Pour~$f$ dans~$\kk'[T]$, la formule~\eqref{eq27} définit encore une fonction convexe de~$\lambda$. Il en résulte que pour~$f$ dans~$\kk[G]$ et~$g$ dans~$G(\kk')$, la composée~$\abs{f}\circ((\Phi\circ\theta_T)g)$ est une fonction convexe sur~$\Lambda$. Ainsi l'application~$(\Phi\circ\theta_T)g$ est continue.

\subsubsection{}\label{sectionD42} Étendons maintenant cette construction aux éléments~$p$ de~$G^\an$.

Soit un point~$p$ de~$G^\an$; ce point est dans l'adhérence de l'image de~$G(\kk_s)$ dans~$G^\an$. Nous allons alors construire l'application~$(\Phi^\an\circ\theta_T )\cdot p$ comme limite simple de fonctions
de la forme~$(\Phi^\an\circ\theta_T )\cdot g$, avec~$g$ dans~$G(\kk_s)$, lorsque l'image de~$g$ dans~$G^\an$ tend vers~$p$.

Montrons que cette limite simple existe et est unique.

\begin{proof}
Notons~$G_{\kk_s}$ le groupe groupe algébrique affine sur~$\kk_s$ d'algèbre~$\kk[G]\tens\kk_s$, notons~$T_{\kk_s}$ son tore sur~$\kk_s$ déployé d'algèbre~$\kk[T]\tens\kk_s$, et~$\Phi_{\kk_s}$ le morphisme~$G_{\kk_s}\to T_{\kk_s}$ issu de~$\Phi$. Soit~$\lambda$ dans~$\Lambda$, soit~$f$ dans~$\kk[G]$ et, pour tout~$g$ dans~$G(\kk_s)$, notons~$f\circ(\Phi_{\kk_s}\cdot g)=\sum_{\chi\in X(T)}a_\chi(g)\chi$ la fonction de~$\kk_s[X(T)]$ qui s'obtient en translatant~$\Phi_{\kk_s}:G_{\kk_s}\to T_{\kk_s}$ par~$g$ puis en composant par~$f$. Tout revient à montrer que lorsque l'image
de~$g$ converge dans~$G^\an$, le nombre réel
\[
\max_{\chi\in X(T)}\abs{a_\chi(g)}\cdot\abs{\varpi}^{\langle\chi,\lambda\rangle}
\]
tend vers une valeur limite.

Par définition,~$\abs{h(g )}$ tend vers une valeur limite pour toute fonction régulière~$h$ sur~$G$ définie sur~$\kk$. Il suffit donc de montrer que sauf pour un nombre fini de caractères~$\chi$, les fonctions 
$g\mapsto a_\chi(g)$ sont nulles et que, pour tout caractère~$\chi$ de~$T$, la formule~$g\mapsto a_\chi(g)$ définit une fonction régulière. Cela résulte de ce que l'action de~$G(\kk_s)$ sur~$\kk_s[G]$ est union de sous-espaces stables sous~$\mathrm{Aut}(\kk_s/\kk)$ de dimension finie, et que l'action de~$G(\kk_s)$ sur un tel sous-espce provient d'une représentation de~$G$ définie sur~$\kk$, \cite[1. \S1 1.9]{Bor91}.
\end{proof}

Comme une limite simple de fonctions convexes est convexe, pour tout~$f$
dans~$\kk[G]$, la composée~$\abs{f}\circ((\Phi^\an\circ\theta_T )\cdot p)$ est une fonction convexe sur~$\Lambda$. En particulier c'est une fonction continue. Par conséquent~$(\Phi^\an\circ\theta_T )\cdot p$ est continue

\subsection{Un Critère}\label{sectionD5} Pour toute extension finie~$\kk'$ de~$\kk$, notons~$G_{\kk'}$ la variété algébrique affine sur~$\kk'$ associée à l'algèbre de type fini~$\kk[G]\tens\kk'$, et~${G_{\kk'}}^\an$ l'espace analytique correspondant. Notons que~$G(\kk')$ est naturellement un \emph{groupe} algébrique affine sur~$\kk'$: son foncteur de points s'écrit~$G_{\kk'}(A)=G(A)$ pour une~$\kk'$-algèbre~$A$.

L'application~$\kk[G]\to\kk[G]\tens\kk'$ induit, par restriction des semi-normes, une application continue~${G_{\kk'}}^\an\to G^\an$. Ces applications sont manifestement compatibles aux morphismes d'extensions.

Notons que le morphisme~$\Phi:T\to G$ induit un morphisme~$\Phi_{\kk'}:T_{\kk'}\to G_{\kk'}$, d'où une application continue~${\Phi_{\kk'}}^\an:{T_{\kk'}}^\an\to {G_{\kk'}}^\an$. Le relèvement~$\Lambda\to{T_{\kk'  }}^\an$ de~$\theta_T:\Lambda\to T^\an$, obtenu en étendant la formule~\eqref{eq27}, est lui aussi compatibles aux extensions.
\begin{proposition} \label{propD1}\label{PropD1}
Soit~$\theta:\Lambda\to G$ une application continue. Supposons que pour toute extension finie~$\kk'$ de~$\kk$, il existe~$\theta_{\kk'}:\Lambda\to {G_{\kk'}}^\an$ telle que
\begin{enumerate}[label=\alph*.]
\item \label{D11} $\theta$ soit la composée de~$\theta_{\kk'}$ avec l'application~${G_{\kk'}}^\an\to G^\an$ ci-dessus;
\item \label{D12}	pour tout~$\mu$ dans~$T(\kk')$, l'action de~$T(\kk')$ par translation à gauche sur la fonction~$\theta_{\kk'}$ commute correspond à son action sur~$\Lambda$:
\begin{equation}
\forall\mu\in T(\kk'),~\theta_{\kk'}(x)\cdot\Phi(\mu)=\theta_{\kk'}(x-\lambda(\mu))
\end{equation}
\end{enumerate}
Alors~$\theta$ est l'application~$(\Phi^\an\circ\theta_T)\cdot p$ où~$p$ est le point~$\theta(0)$. 
\end{proposition}
\begin{proof}
Par limite simple on peut supposer que~$p$ est dans l'image
de~$G(\kk_s)$, puis quitte à considérer une extension finie, que~$p$ appartient à~$G(\kk)$, quitte à faire agir~$p$ à droite, que~$p$ est l'élément neutre.

Par continuité de~$\theta$ et~$\theta_T$ il suffit de montrer~$\theta(\lambda) = (\Phi^\an\circ\theta_T (\lambda))$ pour un ensemble
dense de~$\lambda$ dans~$\Lambda$.

Or l'orbite dans~$\Lambda$ de~$0$ sous l'action de~$T (\kk_s)$ est dense.
Cela provient de la description de cette action pour tout corps local contenu dans~$\kk_s$, grâce à la section~\ref{sectionD31} et au fait que quitte à considérer des racines de l'uniformisante d'ordre premier à la caractéristique, on peut approcher tout nombre réel positif par la valeur absolue d'éléments de~$\kk_s$.

Enfin~$\theta$ et~$(\Phi^\an \circ \theta_T (\lambda))$ concordent en~$0$ et vérifient la loi de transformation (cf. section~\ref{sectionD31}).
\end{proof}
Des sections~\ref{sectionD31} et~\ref{sectionD4} on déduit que cette Proposition~\ref{propD1} a pour corollaire le suivant.
\begin{corollaire}\label{coroD2}
Soit~$\theta:\Lambda\to G^\an$ une application telle que dans la Proposition~\ref{propD1}. Alors pour toute fonction régulière~$f$ sur~$G$, la fonction réelle~$\abs{f}\circ\theta$ est logarithmiquement convexe.
\end{corollaire}
\section{Immeubles euclidiens de Bruhat-Tits}

\subsection{Avant-propos} L'\emph{immeuble euclidien} de Bruhat-Tits d'un groupe algébrique semisimple~$G$ sur un corps local ultramétrique~$\kk$ est un analogue, dans le contexte ultramétrique, de l'espace riemannien symétrique des sous-groupes compacts maximaux associé à un groupe de Lie semi-simple réel~$L$, qui s'écrit~$L/K$, pour un sous-groupe compact maximal~$K$ de~$L$. Dans le contexte ultramétrique, l'analogue naïf du théorème de Cartan ne vaut plus:~$G(\kk)$ ne contient en général pas, à conjugaison près, d'unique sous-groupe compact maximal. Toutefois, du
moins pour les groupe déployés, les énoncés 5.3.1 et 5.3.3 de~\cite{Ber90} restituent \emph{a
posteriori} cette facette de l'analogie entre immeubles et espaces symétriques.

\subsection{Propriétés}
Soit~$\kk$ un corps local muni d'une valeur absolue ultramétrique~$\abs{-}$, et soit~$G$ un groupe algébrique semi-simple sur~$\kk$. L'\emph{immeuble euclidien de Bruhat-Tits} de~$G$ sur~$\kk$, ou plus simplement «~immeuble~», désigne un certain espace métrique~$\Ik(G)$ muni d'une action fidèle proprement continue du groupe topologique~$G(\kk)$ (à gauche, par isométries). L'immeuble est uniquement défini à unique isométrie près. (cf~\cite[2.1]{Tit79}). Le stabilisateur
de l'immeuble est réduit au centre de~$G$. Les stabilisateurs des points de~$\Ik(G)$ sont des sous-groupes dits \emph{parahoriques} de~$G(\kk)$; 
ils sont compacts et ouverts dans~$G(\kk)$: suivant~\cite[Introduction]{BT84} (voir aussi~\cite[3.4.1]{Tit79}), ce sont des groupes de la forme~$G(\mathscr{O}_\kk)$ pour certaines formes entières de~$G$ sur l'anneau des entiers ultramétriques~$\mathscr{O}_\kk$ («~schémas
en groupes plats prolongeant G~»).

L'immeuble~$\Ik(G)$ admet une famille distinguée de parties, appelées \emph{appartements}, réunissant les propriétés suivantes.
\begin{enumerate}[label=\alph*.]
\item \label{a.} L'ensemble des appartements est stable sous l'action de~$G(\kk)$, et l'action de~$G(\kk)$ sur l'ensemble des appartements est transitive.
\item \label{b.} Les appartements sont isométriques à un espace vectoriel euclidien. En particulier, pour tout appartement~$A$, et tout~$g$ dans~$G(\kk)$, l'isométrie~$A\xrightarrow{} g A$ est une application affine.
\item \label{c.} L'immeuble~$\Ik(G)$ est réunion de ses appartements, et tout point a un voisinage formé d'une réunion finie d'appartements.
\item \label{d.} Le stabilisateur dans~$G(\kk)$ d'un appartement donné agit via un réseau du grou\-pe d'isométries de cet appartement.
\item \label{e.} Deux points quelconques de~$\Ik(G)$ sont contenus dans un appartement commun \cite[7.14.18]{BT72}.
\item \label{f.} L'immeuble~$\Ik(G)$ un espace de Hadamard : l'inégalité CAT(0) est satisfaite \cite[3.2]{BT72}.
\item \label{g.} L'immeuble~$\Ik(G)$ admet une structure polysimpliciale~$G(\kk)$-invariante, dont un polysimplexe est domaine fondamental et intersection d'appartements. Le stabilisateur d'un appartement donné préserve un pavage issu d'une structure polysimpliciale qui s'étend en une structure polysimpliciale invariante
sur~$\Ik(G)$.
\end{enumerate}

\subsection{Avertissement}
Nous nous reposons sur~\cite{RTW09} pour la construction
de l'immeuble de Bruhat-Tits. Les hypothèses de travail de ces auteurs sont
énoncées en~\cite[1.3.4]{RTW09}, numéro qui, par ailleurs, indique explicitement que
ces hypothèse sont vérifiées si le corps de base~$\kk$ est un corps local, ce qui est
le cas considéré ici. Une autre référence dans le cas des corps locaux est~\cite{Tit79}.
Cette dernière, reposant sur l'exposé axiomatique~\cite{BT72}, indique, en~\cite[1.5]{Tit79},
quelques réserves sur la satisfiabilité des hypothèses de~\cite{BT72}, qui sont ramenées,
en ce qui concerne~\cite{Tit79}, aux propriétés 1.4.1 et 1.4.2 de~\cite{Tit79}. La
suite~\cite{BT84} de l'exposé~\cite{BT72}, suite postérieure à la référence~\cite{Tit79}, démontre
que les hypothèses de~\cite{BT72} sont satisfaites « pour tout groupe réductif sur un
corps de valuation discrète hensélien à corps résiduel parfait » (Introduction,
page 9).

Nous esquissons la construction de l'immeuble indiquée dans~\cite{RTW09}. Que
cette construction vérifie les propriétés indiquées plus haut résulte, pour certaines
de ces propriétés, de la construction même, qui procède par analyse-synthèse.
Pour les autres propriétés cela résulte d'une part de la satisfiabilité des hypothèse
de travail de~\cite{BT72} pour les corps locaux pour laquelle nous venons d'indiquer
des références ; d'autre part de certaines conclusions de~\cite{BT72}; enfin, dans le cas
des corps locaux notamment, de la référence~\cite{Tit79}.

\subsection{Construction}\label{sectionE4} Comme tout appartement A de l'immeuble~$\Ik(G)$ est une partie génératrice, on peut construire~$\Ik(G)$ comme quotient de~$G(k) \times A$. Les relations par lesquelles on quotiente sont engendrées par celles qui définissent le stabilisateur de chaque point de~$A$, et celles qui déterminent l'action sur~$A$ du stabilisateur
de~$A$. C'est l'approche utilisée dans~\cite[1.3]{RTW09}.

Ne discutons que du modèle de A muni de l'action de son stabilisateur.
Fixons un tore déployé maximal~$T$ de~$G$, et posons~$A = \Lambda = \Hom(\GL(1),T ) \tens \R$.
Le normalisateur~$N(T )(\kk)$ de~$T$ dans~$G(\kk)$ agit par transport de structure sur~$\Lambda = \Hom(\GL(1),T )\tens\R$, et le noyau de cette action est le centralisateur~$C(T )(\kk)$
de~$T$ dans~$G(\kk)$. Le groupe quotient~$N(T )(\kk)/C(T )(\kk)$ est un groupe fini, le groupe de Weyl sphérique de~$G$ relatif à~$T$. Alors le stabilisateur de l'appartement~$A$ dans~$G(\kk)$
est~$N(T )(\kk)$ et l'action de~$N(T )(\kk)$ est une action affine
\begin{itemize}
\item dont la partie linéaire est l'action précédente, 
\item et pour laquelle~$T (\kk)$, qui est contenu dans~$N(T )(\kk)$, agit par~\ref{sectionD5}.
\end{itemize}
\section{{Convexit\'{e}}}\label{sectionF}
\subsection{{Notion de convexit\'{e}}}
Rappelons qu'un~\emph{appartement} de~$\Ik(G)$ est l'image d'une partie de la forme~$\{g\}\times\Lambda$ par l'application~$G(\kk)\times\Lambda\rightarrow \Ik(G)$. Les appartements sont permutés transitivements sous l'action de~$G(\kk)$. En outre, le stabilisateur d'un appartement~$A$ agit sur~$A$ de manière affine et cette action contient l'action additive d'un réseau vectoriel de~$\Lambda$ (voir \cite[1.2, 1.3]{Tit79}).

\subsubsection{}\label{sectionF11}
En particulier, si~$F$ est un parallélotope fondamental de ce réseau,~$F$ est une partie bornée qui rencontre tout orbite de~$G(\kk)$ rencontrant cet appartement, ce qui est le cas de toute orbite de~$G(\kk)$. L'immeuble~$\Ik(G)$ contient donc une partie génératrice compacte.

Rappelons que deux points quelconques~$x$ et~$y$ de~$\Ik(G)$ sont contenus dans un appartement commun~(\cite[7.14.18]{BT72}). on peut donc définir le segment~$[x;y]$ joignant~$x$ à~$y$ dans~$A$. Comme le stabilisateur d'un appartement~$A$ agit de manière affine sur~$A$, ni le segment~$[x;y]$ ni la structure affine sur~$[x;y]$ ne dépendent de l'appartement choisi (cf.~\cite[2.2.1]{Tit79}).

\label{convexité}
Une fonction réelle continue sur~$\Ik (G)$ dont la restriction à tout segment
est \emph{con\-ve\-xe} (resp. \emph{affine}), sera dite convexe (resp. affine). Il revient au même
de dire que la restriction à chaque appartement est convexe (resp. affine),
relativement à la structure affine de cet appartement.

Bien évidemment, les fonctions affines et les fonctions dont le logarithme est
affine sont convexes. En outre, toute fonction réelle qui s'écrit comme borne
supérieure ou limite simple de fonctions convexes est convexe.

\label{convexite2}
Une partie de~$\Ik(G)$ sera dite convexe si elle contient tout segment dont elle contient les extrémités.

\subsection{Une décomposition de l'immeuble}
Soit~$k$ un sous-groupe compact de~$G(\kk)$. Le groupe compact~$k$ agit par isométries sur l'immeuble~$\Ik(G)$.
Notons~$\Ik(G)^k$ le lieu fixe de l'action de~$k$, et~$Z_G(k)$ le centralisateur de~$k$ dans~$G$.
Alors~$\Ik(G)^k$ est stable sous l'action du centralisateur~$Z_G(k)(\kk)$ de~$k$ dans~$G(\kk)$.
Nous allons montrer le premier énoncé suivant.

\begin{theoreme}\label{TheoF1}\label{propF1}
Soit~$k$ un sous-groupe compact de~$G(\kk)$ dont l'adhérence
de Zariski dans~$G$ est un sous-groupe fortement réductif, et soit~$Z_G (k)$ le centralisateur de~$k$ dans~$G$. Alors il existe une partie compacte non vide~$C$ de~$\Ik(G)^k$ qui est génératrice pour l'action de~$Z G (k)(\kk)$: on a~$\Ik(G)^k=Z_G (k)(\kk)\cdot C$.
\end{theoreme}
Commençons par un lemme. 
\footnote{N.d.É.: Ce lemme gagne à être mieux connu. Il permet par exemple de simplifier l'argumentation finale de~\cite[]{UllmoYaffaev} en coupant court à leur introduction de mesures.}

\subsection*{\emph{Erratum}}\label{Erratum} -- La version initiale de ce Lemme, supposant que l'on a « dans~$G$~[...]~un sous-groupe réductif~» est erronée (comme pointé par E.~Breuillard). Une confusion a eu lieu sur la terminologie «~\emph{reduced}~» de \cite[p. 101]{PR94}, qui n'est pas le sens usuel algébrico-géométrique. Plutôt que la notion utilisée par~\cite[p.~101]{PR94}, la notion idoine est celle de sous-groupe «~\emph{strongly reductive in~$G$}~», basée sur le travail~\cite[\S16]{Richardson} de Richardson. Ce dernier travaille toutefois sur des corps algébriquement clos, d'a\-près~\cite[\S1.1]{Richardson}. 

-- Cette correction affecte le Théorème~\ref{TheoF1}, la Proposition~\ref{Prop21} et la première conclusion du Théorème~\ref{Theo11}.

-- En caractéristique~$0$ on retombe sur la notion usuelle de groupe réductif; voir à ce sujet~\cite[Proposition~4.2]{SerreCR}. Il n'y a pas contagion de l'erreur de ce côté là. C'est le cas des applications mentionnées en dynamique.

-- Cette dernière référence contient une revue détaillée de la notion de «~strong reductivity in~»\footnote{Plutôt que ``«strict reductivity»'' que l'on lit en \cite[page~932-13]{SerreCR}.} rebaptisée (une fois généralisée à un corps~$k$ quelconque)~«~$G$-cr~», pour «~complète réductibilité~». Pointons~\cite[\S3.1.1, Th~3.5, Th.~3.7, \S~4.1]{SerreCR}.

-- Rappelons que~$H$ est dit «~réductif dans~»~$G$ au sens de~\cite{RichardShah} si l'action de~$H$ sur~$\lie{g}$ est semi-simple. Alors~$(HZ_G)\cap G^\text{der}$ est «~\emph{reduced subgroup}~» de~$G$ au sens de~\cite{PR94}. Il est donc \emph{strongly reductive in~$G$} d'après~\cite[Th.~2.16]{PR94} et~\cite[Th.16.4]{Richardson}. Comme le notion \emph{strongly reductive} s'écrit en termes des sous-groupes paraboliques contenant~$H$, elle est insensible au passage de~$H$ à~$HZ_G$ puis à~$HZ_G\cap G^\text{der}$, le sous-groupe~$H$ lui-même est \emph{stronlgy reductive in~$G$}.

\begin{lemme}\label{LemmeF2}\label{lemmeF2}\footnote{N.d.É.: Bien que l'on se réfère à~\cite{PR94}, ce lemme est basé sur le travail~\cite{Richardson} de Richardson, algébrique, dont c'est une variante, et conséquence, ultramétrique.}
Soit~$k$ un sous-groupe compact de~$G(\kk)$ dont l'adhérence
de Zariski dans~$G$ est un sous-groupe fortement réductif dans~$G$,  et soit~$Z_G (k)$ le centralisateur de~$k$ dans~$G$.

Pour tout compact~$C$ de~$G(\kk)$, il existe un compact~$C'$ de~$G(k)$ tel que le transporteur
\begin{equation}
T (k, C) = \left\{g \in G(k)\middle|gkg^{-1} \subseteq C\right\}
\end{equation}
de~$k$ dans~$C$ s'écrive~$C'\cdot Z_G (k)(\kk)$.

\end{lemme}
\begin{proof}
 Tout d'abord ce transporteur est fermé. Pour chaque~$x$ dans~$k$ l'application~$g \mapsto g xg^{-1}$ est continue, l'image inverse de~$C$ est fermée; ce transporteur est l'intersection de ces images inverses, des fermés. Tout revient donc à montrer que~$T (k,C )$ est le
saturé par~$Z_G (k)(\kk)$ d'une partie relativement compacte. Comme~$T (k,C )$ est manifestement invariant par~$Z_G (k)(\kk)$, il suffit de montrer qu'il est contenu dans le saturé d'une partie compacte.

Remarquons que, pour la topologie de Zariski,~$k$ est noethérien. Par con-
séquent, pour cette topologie, il est topologiquement de type fini. En particulier, il existe une partie finie~$\{x_1 ; \ldots ; x_n \}$ topologiquement génératrice de~$k$ pour la topologie de Zariski. Le centralisateur de cette partie est le centralisateur de~$k$, c'est-à-dire~$Z_G (k)$.

Appliquant~\cite[Theorem 16.4]{Richardson}, on montre que l'application
\[
g\mapsto (gx_1g^{-1};\ldots;gx_ng^{-1})
\]
induit une immersion fermée de~$G/Z_G(k)$ dans~$G^n$. Par conséquent l'application correspondante
\[
\phi:\left(G/Z_G(k)\right)(\kk)\to G(\kk)^n
\]
est~\emph{propre}. En particulier l'image inverse de~$C^n$ dans~$\left(G/Z_G(k)\right)(\kk)$ est compacte. Or cette image inverse de~$C^n$ contient l'image de~$T(k,C)$ par~$G(\kk)\to\left(G/Z_G(k)\right)(\kk)$.

Il suffit donc de montrer que tout compact de~$\left(G/Z_G(k)\right)(\kk)$ rencontre l'image par~$\phi$ de~$G(\kk)$ en l'image d'un compact. D'après~\cite{PR94}, l'application~$\phi$ est ouverte, et les orbites de~$G(\kk)$ dans~$\left(G/Z_G(k)\right)(\kk)$ sont toutes ouvertes; elles sont donc aussi fermées. En particulier l'image~$\phi(G(\kk))$ est fermée, et intersecte donc tout compact en un compact. Il suffit de montrer que tout compact de~$G(\kk)/Z(\kk)$ est l'image d'un compact. D'après la propriété de Borel-Lebesgue, il suffit de travailler localement (c.à.d. montrer que tout ouvert assez petit est contenu dans l'image d'un compact). Or~$\phi$ est ouverte et~$G(\kk)$ est localement compact.
\end{proof}

Avant de démontrer le Théorème~\ref{TheoF1}, rappelons quelques faits, dont certains
sont bien connus.
\begin{enumerate}[label=\alph*.]
\item \label{a.} Il existe un compact~$F$ de~$\Ik(G)$ rencontrant toute orbite de~$G(\kk)$ (cf.~\ref{sectionF11}).
\item \label{b.} Pour toute partie bornée non vide~$P$ de~$\Ik(G)$, le stabilisateur de~$P$ dans~$G(\kk)$ est un sous-groupe compact et ouvert (cf.\cite[3.2]{Tit79}, \cite[Introduction]{BT72}).
En particulier, pour tout point~$p$ de~$\Ik (G)$, le fixateur de~$p$
dans~$G(k)$ est un sous-groupe compact et ouvert.
\item \label{c.} Le fixateur commun à tous les éléments d'une partie bornée non vide est
compact et ouvert (cf. \emph{supra}).
\item \label{d.} Pour tout sous-groupe ouvert~$U$ de~$G(\kk)$, le lieu fixe de~$U$ dans~$G$ est une partie compacte de~$\Ik(G)$.
\item \label{e.} Tout sous-groupe compact de~$G(\kk)$ est contenu dans un sous-groupe compact maximal.
\item \label{f.} Les sous-groupes compacts maximaux de~$G(\kk)$ sont ouverts.
\item \label{g.} Ils forment un nombre fini de classes de conjugaison (\cite[3.3.3]{BT72}).
\item \label{h.} Tout sous-groupe compact ouvert de~$G(\kk)$ est contenu dans un nombre fini de sous-groupe compact maximaux.
\end{enumerate}
\begin{proof}[Démonstration du Théorème~\ref{TheoF1}]
D'après le point~\ref{a.}, il existe un compact~$F$ de~$\Ik (G)$ rencontrant toute orbite de~$G(\kk)$. D'après le point~\ref{b.}, le stabilisateur commun à tous les points de~$F$ est un sous-groupe \emph{compact et ouvert} de~$G(\kk)$. Notons-le~$K_F$.

Pour tout point~$f$ de~$F$ le stabilisateur de~$f$, disons~$K_f$, est également un sous-groupe compact et ouvert de~$G(\kk)$ (point~\ref{b.} ci-dessus). Par construction ces groupes contiennent~$K_F$. Appliquant le point~\ref{h.}, il s'ensuit que l'ensemble~$E_F=\{K_f|f\in F\}$ de sous-groupes compacts et ouverts de~$G(\kk)$ est un ensemble \emph{fini}.

Pour~$K$ dans~$E_F$, notons~$F_K = \{ f \in F |K_f = K \}$. Par construction de~$E_F$, le compact~$F$ s'écrit comme l'union finie~$F = \bigcup\{F_K |K \in EF \}$. Notons que chaque~$F_K$, étant
contenu dans~$F$, est borné.

Soit~$k$ le groupe compact mentionné dans l'énoncé du théorème. Pour~$g$
dans~$G(\kk)$ et~$f$ dans~$F_K$, le point~$g\cdot f$ de~$\Ik (G)$ est fixé par~$k$ si et seulement si~$g^{-1}kg$ est contenu dans le stabilisateur de~$f$, c'est-à-dire dans~$K$.  Autrement dit~$g^{-1}$ est
dans le transporteur~$T (k, K )$ de~$k$ dans~$K$. D'après le Lemme~\ref{lemmeF2}, il existe un compact~$C_K$ de~$G(\kk)$ tel que~$T(k,K)$ s'écrive~$C_K\cdot Z_G(k)$. Par conséquent, l'élément~$g\cdot f$ ci-dessus appartient à la partie~$Z_G(k)(\kk)\cdot {C_K}^{-1}\cdot F_K$ du lieu fixe~$\Ik(G)^k$ de~$k$ agissant sur l'immeuble~$\Ik(G)$.

On a montré que tout point de la forme~$g\cdot f$, avec~$g$ dans~$G(\kk)$ et~$f$ dans~$F_K$ appartient en fait à~$Z_G(k)(\kk)\cdot {C_K}^{-1}\cdot F_K$.

Comme~$F$ rencontre toute orbite de~$G(\kk)$ dans~$\Ik (G)$, tout point~$p$ de~$\Ik (G)$ s'écrit~$f\cdot g$ avec~$f$ dans~$F$ et~$g$ dans~$G(\kk)$.  Ainsi tout point fixe de k appartient
à~$Z_G(k)(\kk)\cdot {C_K}^{-1}\cdot F_K$ pour un certain~$K$ dans~$F$.  Par conséquent,~$\Ik (G)^k$ est contenu dans~$\bigcup_{K\in E_F}Z_G(k)(\kk)\cdot {C_K}^{-1}\cdot F_K$.

Comme~$E_F$ est fini et que les~$F_K$ et~$C_K$ sont bornés, la partie~$\bigcup_{K\in E_F}Z_G(k)(\kk)\cdot {C_K}^{-1}\cdot F_K$ est bornée, donc son adhérence est compacte. Remarquons que~$\Ik(G)$ est fermé et invariant sous l'action de~$Z_G(k)(\kk)$. Par conséquent
\[
C=\overline{\bigcup_{K\in E_F}{C_K}^{-1}\cdot F_K}\cap \Ik(G)^k
\]
est un compact de~$\Ik(G)$ tel que~$\Ik(G)^k=CZ_(G)(k)(\kk)$.

\end{proof}
\subsection{Plongement analytique et Convexité des fonctions régulières sur l'immeuble} Nous basant sur~\cite{RTW09}, démontrons l'énoncé suivant.
\begin{theoreme}\label{TheoF3}
Soit~$\kk$ un corps local muni d'une valeur absolue ultramétrique, soit~$G$ un groupe algébrique semi-simple sur~$\kk$ et fixons un tore
algébrique déployé maximal~$T$ dans~$G$. Notons~$\Phi : T \rightarrow G$ le morphisme
d'inclusion, $\theta_T : \Lambda \rightarrow T^\an$ l'application~\eqref{eq23},~$\phi^\an : T^\an \rightarrow G^\an$ l'application correspondante, et~$\Ik(G)$ l'immeuble de Bruhat-Tits de~$G$ sur~$\kk$. 

Pour tout point~$p$ de~$G^\an$, on note~$((\Phi^\an \circ \theta_T )\cdot p) : \Lambda \rightarrow G^\an$ l'application définie dans la
section~\ref{sectionD4}, et pour tout~$g$ dans~$G(\kk)$ , on note~$g((\Phi^\an \circ \theta_T )\cdot p)$ l'application
translatée.

Alors il existe un point~$p$ de~$G^\an$ tel que l'application de~$G(\kk) \times \Lambda$ donnée
par
\begin{equation}\label{eq29}
(g,\lambda)\mapsto g\cdot ((\Phi^\an\circ \theta_T)\cdot p)(\lambda)
\end{equation}
passe au quotient en une application équivariante à gauche
\[
\theta:\Ik(G)\rightarrow G^\an.
\]

En outre on peut supposer que le stabilisateur à droite~$G_p$ de~$p$ dans~$G(k)$ est compact et ouvert.
\end{theoreme}

Dans le cas déployé, cet énoncé résulte de la construction explicite de V. Ber\-ko\-vich, dans \cite[5.3]{Ber90} et du Théorème 5.4.2 qui s'ensuit.

Dans le cas général, notre référence est~\cite{RTW09}, dont l'approche est différente, et repose sur les propriétés de fonctorialité des immeubles par des
extensions, non nécessairement algébriques, de corps ultramétriques. Pour pouvoir démontrer l'énon\-cé~\ref{TheoF3}, nous allons utiliser le critère~\ref{propD1}.
\begin{proof}
Soit~$\Theta:\Ik(G)\times\Ik(G)\to G^\an$ l'application de~\cite[2.3]{RTW09}.
Il suit de~\cite[Proposition~2.12]{RTW09} que l'application~$\Theta$ est continue et vérifie
\[
	\forall x,y\in \Ik(G),
		\forall g,h \in G(\kk),
				\Theta(gx,hy)=h\Theta(x,y)g^{-1}.
\]
Soit~$o$ dans~$\Ik(G)$ et notons~$\Theta_o$ l'application~$x\mapsto\theta(o,x)$. Alors~$\Theta_o$  est une application continue et équivariante de~$\Ik(G)$ dans~$G^\an$. Montrons que~$\Theta_o$ répond à
l'énoncé.

Par équivariance, il suffit donc de montrer qu'elle s'écrit sous la forme~\eqref{eq29} sur un seul appartement, par exemple~$\Lambda$, de~$\Ik (G)$.

Il suffit donc de montrer que la restriction de~$\Theta_o$ à~$\Lambda$ vérifie le critère~\ref{PropD1}. L'application~$\Theta_o$ est bien continue.
Soit~$\kk'$ une  une extension finie de~$\kk$, et soit~$\Theta_{\kk'}$ l'application composée issue du coin supérieur gauche du carré commutatif de \cite[Proposition 2.12 (ii).]{RTW09}.
Par commutativité, l'application~$(\Theta_{\kk'})_o:x\mapsto \Theta_{\kk'}(o,x)$ de~$\Lambda$ dans~$G_{\kk'}^\an$ répond  à la condition~\ref{D11} de~\ref{PropD1}.
Appliquant la proposition 2.12 de~\cite{RTW09} au corps~$\kk'$, nous obtenons que l'application~$(\Theta_{\kk'})_o$ est équivariante à gauche sur~$\Lambda$.
La seconde condition de la Proposition~\ref{PropD1} découle ainsi de la description de l'identité de l'action de~$T(\kk')$ sur~$\Lambda$, pris comme appartement (cf. section~\ref{sectionE4}), avec l'action donnée en section~\ref{sectionD31}.

\end{proof}
En appliquant le Corollaire~\ref{coroD2}, on en déduit ceci.
\begin{corollaire}\label{coroF4}
Soit~$\Theta$ une application~$\Ik(G) \rightarrow G^\an$ telle que dans le
Théorème~\ref{TheoF3}. 

Alors pour toute fonction régulière~$f$ sur~$G$, la fonction réelle~$\abs{f}  \circ\Theta$ est logarithmiquement convexe sur~$\Ik(G)$.
\end{corollaire}
\bibliographystyle{alpha}  
\bibliography{ultrametrique}
\end{document}